\documentclass[11pt]{article}

\usepackage{latexsym}
\usepackage{graphics}
\usepackage{amsmath}
\usepackage{xspace}
\usepackage{amssymb}
\usepackage{psfrag}
\usepackage{epsfig}
\usepackage{amsthm}
\usepackage{pst-all}

\usepackage{color}

\definecolor{Red}{rgb}{1,0,0}
\definecolor{Blue}{rgb}{0,0,1}
\definecolor{Olive}{rgb}{0.41,0.55,0.13}
\definecolor{Yarok}{rgb}{0,0.5,0}
\definecolor{Green}{rgb}{0,1,0}
\definecolor{MGreen}{rgb}{0,0.8,0}
\definecolor{DGreen}{rgb}{0,0.55,0}
\definecolor{Yellow}{rgb}{1,1,0}
\definecolor{Cyan}{rgb}{0,1,1}
\definecolor{Magenta}{rgb}{1,0,1}
\definecolor{Orange}{rgb}{1,.5,0}
\definecolor{Violet}{rgb}{.5,0,.5}
\definecolor{Purple}{rgb}{.75,0,.25}
\definecolor{Brown}{rgb}{.75,.5,.25}
\definecolor{Grey}{rgb}{.5,.5,.5}

\def\red{\color{Red}}

\def\dgreen{\color{DGreen}}

\newcommand{\mnote}[1]{{\red Madhu's Note: #1}}
\newcommand{\dnote}[1]{{\dgreen David's Note: #1}}

\newcommand{\bfU}{{\bf U}}
\newcommand{\bfu}{{\bf u}}
\newcommand{\bfV}{{\bf V}}
\newcommand{\bfv}{{\bf v}}

\newcommand{\bfZ}{{\bf Z}}
\newcommand{\bfz}{{\bf z}}
\newcommand{\bfPhi}{{\bf\Phi}}

\newcommand{\pr}{\mathbb{P}}
\newcommand{\E}{\mathbb{E}}

\newcommand{\G}{\mathbb{G}}
\newcommand{\F}{\mathbb{F}}

\newcommand{\sign}{\mathrm{sign}}

\newcommand{\ignore}[1]{\relax}

\newtheorem{theorem}{Theorem}[section]

\newtheorem{lemma}[theorem]{Lemma}

\newtheorem{prop}[theorem]{Proposition}
\newtheorem{proposition}[theorem]{Proposition}
\newtheorem{coro}[theorem]{Corollary}

\newtheorem{definition}[theorem]{Definition}

\newtheorem{observation}[theorem]{Observation}

\newcommand{\Lovasz}{Lov\'{a}sz}
\newcommand{\local}{decimation}

\definecolor{Red}{rgb}{1,0,0}
\definecolor{Blue}{rgb}{0,0,1}
\definecolor{Olive}{rgb}{0.41,0.55,0.13}
\definecolor{Green}{rgb}{0,1,0}
\definecolor{MGreen}{rgb}{0,0.8,0}
\definecolor{DGreen}{rgb}{0,0.55,0}
\definecolor{Yellow}{rgb}{1,1,0}
\definecolor{Cyan}{rgb}{0,1,1}
\definecolor{Magenta}{rgb}{1,0,1}
\definecolor{Orange}{rgb}{1,.5,0}
\definecolor{Violet}{rgb}{.5,0,.5}
\definecolor{Purple}{rgb}{.75,0,.25}
\definecolor{Brown}{rgb}{.75,.5,.25}
\definecolor{Grey}{rgb}{.5,.5,.5}
\definecolor{Pink}{rgb}{1,0,1}
\definecolor{DBrown}{rgb}{.5,.34,.16}
\definecolor{Black}{rgb}{0,0,0}

\def\red{\color{Red}}

\def\dgreen{\color{DGreen}}

\author{
{\sf David Gamarnik}\thanks{MIT; e-mail: {\tt gamarnik@mit.edu}.Research supported  by the NSF grants CMMI-1335155.}
\and
{\sf Madhu Sudan}\thanks{Microsoft Research New England; e-mail: {\tt madhu@mit.edu}}
}

\begin{document}

\title{Performance of Survey Propagation guided decimation algorithm for the random NAE-$K$-SAT problem}
\date{}

\maketitle

\begin{abstract}
We show that the Survey Propagation-guided decimation algorithm fails to find satisfying
assignments on random instances of the ``Not-All-Equal-$K$-SAT'' problem if
the number of message passing iterations is bounded by a
constant
independent of the size of the instance and the clause-to-variable ratio is above $(1+o_K(1)){2^{K-1}\over K}\log^2 K$ for sufficiently large $K$.
Our analysis in fact applies to a broad
class of algorithms that may be described as ``sequential local algorithms''.
Such algorithms iteratively set variables based on some local
information and/or local randomness, and then recurse on the reduced
instance. Survey Propagation (SP)-guided as well as Belief Propagation (BP)-guided decimation algorithms --- two  widely studied message passing based algorithms,
fall under this category of algorithms provided the number of message passing iterations is bounded by a constant. Another well-known algorithm falling into this category
is the Unit Clause (greedy) algorithm. Our work constitutes the first rigorous analysis of the performance of the SP-guided decimation algorithm.

The approach underlying our paper is based on an intricate geometry of the solution space of random NAE-$K$-SAT problem. We show that above the
$(1+o_K(1)){2^{K-1}\over K}\log^2 K$
threshold, the overlap structure of $m$-tuples of satisfying assignments exhibit a certain clustering behavior expressed in the form of some constraints
on pair-wise distances between the $m$ assignments, for appropriately chosen positive integer $m$. We further show that if a sequential local algorithm succeeds in finding
a satisfying assignment with probability bounded away from zero, then one can construct an $m$-tuple of solutions violating these constraints, thus
leading to a contradiction. Along with~\cite{gamarnik2013limits}, where a similar approach was used by the authors in a (somewhat simpler) setting of non-sequential local algorithms,
this result is the first work which directly links the clustering property of random constraint satisfaction problems to the computational hardness of
finding satisfying assignments.
\end{abstract}

\section{Introduction}

In this work we study the behavior of some ``natural'', statistical-physics-motivated,
algorithms for constraint satisfaction problems
on random instances. These algorithms, specifically BP-guided and SP-guided decimation algorithms, exhibited a spectacular performance empirically,
capable of finding solutions very rapidly and very close to the thresholds, beyond which the satisfying assignments do not exist or are conjectured not to exist.
A partial list of references documenting the performance of these algorithms includes the following papers~\cite{mezard2002analytic},\cite{braunstein2005survey},\cite{krzakala2007gibbs},
\cite{ricci2009cavity},\cite{dall2008entropy},\cite{kroc2012survey} as well as the book by Mezard and Montanari~\cite{MezardMontanariBook}.
At the same time, mathematically rigorous analysis of these algorithms is mostly lacking. Notable exceptions are the works of Coja-Oghlan~\cite{coja2011belief}
who analyzed the perforamnce of the BP-guided decimation algorithm
for random K-SAT problem, and of Maneva et al.~\cite{maneva2007new} who reformulate Survey Propagation algorithm as the Belief Propagation algorithm on a ''lifted''
Markov Random Field. No rigorous results on the performance of the SP-guided algorithm is available, to the best of our knowledge.

\subsection{Our setting and results}

In this work we consider a class of algorithms which we dub ``sequential local algorithms'' that capture natural local implementations of BP-guided and SP-guided decimation algorithms. We analyze their behavior on random instances of
``Not-All-Equal-K-SAT (NAE-$K$-SAT)''. We describe the NAE-$K$-SAT problem,
and our class of algorithms in that order below.

The NAE-$K$-SAT problem
a Boolean constraint satisfaction problem closely related to more commonly studied K-SAT problem.
An instance of the NAE-$K$-SAT problem consists of a collection of $m$
$K$-clauses on $n$ Boolean variables $x_1,\ldots,x_n$.
Each $K$-clause
is given by $K$-literals, where each literal is either one of the variables
or its negation. The clause is satisfied by a Boolean assignment to the
variables if at least one of the literals is satisfied (set to $1$) and at
least one is unsatisfied (set to $0$).
(This symmetry between satisfied and unsatisfied literals lends a convenient symmetry to the NAE-$K$-SAT problem that is not shared by the K-SAT counterpart).

In this work we consider the ability to find satisfying assignments to
random instances of the NAE-$K$-SAT problem. Here the $m$ clauses are chosen
uniformly and independently from the collection of $2^K\cdot {n \choose
K}$ possible $K$-clauses. In particular we consider the setting where $m =d\cdot n$ for some constant $d = d(K)$ which depends on $K$, but not $n$, and
consider what is the largest $d$ for which there exists a efficient algorithm for identifying a
satisfying assignment with probability approaching one as $n\rightarrow\infty$. The parameter $d$ is often referred to as the clause density.
Of course, no algorithm can find a satisfying
assignment if none exists; and the limit of when such an assignment exists
is well-studied.
In particular Coja-Oghlan and and Panagiotou~\cite{coja2012catching}
have established that random instances of the NAE-$K$-SAT problem are
satisfiable w.h.p. when the density $d$ is below
$d_s\triangleq 2^{K-1} \ln 2 -\ln 2/2-1/4-o_K(1)$, and is not satisfiable w.h.p. when $d>d_s$,
Here $o_K(\cdot)$ denotes a function converging to zero as $K$ increases. (A similar convention is adopted for other notations for orders of magnitude).
Our interest is in determining how qualitatively close to this threshold an efficient algorithm can get,
i.e., how does the largest density at which the algorithm manages to find satisfying
assignments compare with $d_s$.

The class of algorithms that we explore in this work are what we call
``sequential local algorithms''. This is a class that abstracts algorithms
such as the
BP- and SP-guided decimation algorithms when the number of message
passage iterations used in every decimation step is bounded by a
constant $r$, independent of the size of the instance. (BP- and SP-guided decimation algorithms really form a very general class with many possible implementations and interpretations. In Section~\ref{subsection:SPEmpirical}
we discuss the specific assumptions we make and their
potential limitations.)
A sequential local algorithm can be described roughly as follows. The algorithm works by assigning Boolean values to
variables sequentially, where a chosen variable is assigned its value by a
potentially probabilistic choice, which depends on the local neighborhood
of the variable at the time the choice is made. The local neighborhood is defined to be the graph-theoretic $B(r)$ ball of constant $r$
radius with respect to the underlying factor graph on the set of variables and clauses, to be defined later.
Once a variable is assigned a value, the formula is simplified
(removing some clauses, and restricting others). This in turn may
influence the local neighborhoods of other variables, and when the future variables are set to particular Boolean values,
this is done with respect to thus possibly modified neighborhoods. The algorithm
continues with its iterations till all variables are set.
In the specific context of  BP-guide decimation algorithm based on $r$ iterations, the local rule  assigns value $1$ to a variable $x$
with probability equal to the
fraction of assignments in which $x$ is assigned value $1$
among all assignments that
satisfy all clauses in the local neighborhood $B(r)$.
The SP-guided decimation algorithm uses a more complex rule for its assignments. It is based on lifting the Boolean constraint satisfaction problem
to a constraint satisfaction problem involving three decisions, as opposed to two decisions, but otherwise follows the same spirit.

Our main contribution (Theorem~\ref{theorem:MainResultNAE-$K$-SAT})
is to show that, with high probability (w.h.p.)
as the size of the instance diverges to infinity,
every ``balanced'' sequential local algorithm fails to produce a satisfying assignment when the ratio $d$ of the number of clauses to the number of variables exceeds
$(1+o_K(1)){2^{K-1}\over K}\log^2 K$ and clause size $K$ is sufficiently large (but independent from the number of variables).
``Balance'' is a technical condition explained in Definition~\ref{defn:balance},
which says that the local algorithm respects the inherent symmetry between
$0$ and $1$. It is a condition satisfied by all known algorithms inlcuding
BP- and SP-guided decimation.

Our bound on the ratio $d$ is reasonably close to bounds at which simple
algorithms actually work.
In particular, it is well-known that a very simple Unit Clause algorithm is capable of finding satisfying assignments for this problem when $d$ is below
$\rho{2^{K-1}\over K}$ for some universal constant $\rho$,~\cite{achlioptas2002two} for $K$ sufficiently large.
The Unit Clause algorithm is a special case of the sequential local algorithm,
as we will show in the paper, and is the best known algorithm for this problem. (A better algorithm is known for the random K-SAT problem which
works up to clause to variables density $(1-o_K(1)){2^{K}\over K}\log K$~\cite{coja2010better}. It is likely that a similar idea can be applied to
the NAE-K-SAT setting, but such a result is not available to the best of our knowledge).
One of the hopes was that BP- and SP-guided decimation algorithms may be
able to bridge this factor of $K$ between the unit clause algorithm and the
satisfiability threshold $d_s$ above.
Our result however implies that, short of possibly a $\log^2 K$
multiplicative factor, the ''infamous'' factor $O(K)$ gap between the satisfiability threshold
and the region achievable by known algorithms cannot be  broken by means of sequential local
algorithms, in particular by BP- and SP-guided decimation algorithms with constant number of
rounds of message passing iterations.

Previously, Coja-Oghlan
~\cite{coja2011belief}
showed that the BP-guided decimation algorithm fails to
find satisfying assignments for random K-SAT problem when $d\ge \rho {2^K\over K}$ for some
universal constant $\rho$, for an arbitrary number of iterations $r$, which in particular might depend on the number of variables.
(Here $2^K$ factor is an ''appropriate'' substitution for $2^{K-1}$ when switching from NAE-$K$-SAT to the K-SAT problem. We maintain this distinction, even though
technically it is eliminated by constant $\rho$). It is reasonable to expect that his result holds also for NAE-$K$-SAT problem using the same analysis.
Thus our result reproduces the main result of~\cite{coja2011belief} but only in the special case of bounded number of iterations (short of additional $\log^2 K$ factor).
At the same time, however, our result is applied in a ''blanket way'' to a broad class of algorithms, including most notably SP-guided decimation algorithm,
and, unlike the analysis of~\cite{coja2011belief}, our analysis is insensitive to the details of the algorithm.

\subsection{Techniques}

Our main proof technique relies on the intricate geometry of the solution space of the random NAE-$K$-SAT problem. Specifically it relies on the so-called
\emph{overlap} structure of satisfying assignments of random NAE-$K$-SAT, which was earlier established for random $K$-SAT problem, and several other related problems,
including the problem of proper coloring of sparse random graphs.
Roughly speaking, the property says that, above a certain density,
the Hamming distance between every pair of satisfying assignments, commonly called \emph{overlap} in statistical physics literature,
normalized by the number of variables, is either smaller than a certain constant $\delta_1$
or larger than some constant $1\ge \delta_2>\delta_1$. As as result the solutions can be grouped into different subsets (clusters) based on their proximity to each other.
For the case of NAE-$K$-SAT problem this 2-overlap property can be established for densities $d$ exceeding approximately $d>d_s/2$. (A weaker
version of this result corresponding to "almost" all pairs does hold at densities above $O\left({d_s\over K}\log K\right)$~\cite{montanari2011reconstruction}).
Unfortunately, this is not strong enough to cover the regime of $d>(d_s/K)\log^2K$ claimed in our main theorem, so instead we have
to establish a certain property regarding $m$-overlaps of satisfying assignments,
for appropriately chosen $m$. This is the essence of Theorem~\ref{theorem:StrongClustering} which we prove in this paper.
Roughly speaking this theorem says that when $d\ge (1+\epsilon){d_s\over K}\log^2K $, and $K$ is sufficiently large,
one cannot find $m\approx \epsilon K/\log K$ satisfying assignments such
that the Hamming distance (overlap) between every pair of the assignment normalized by the number of variables is  $\approx\log K/K$.
We then show that for every $\beta\in (0,1)$, if a sequential local algorithm was capable of finding a satisfying assignment, with probability bounded away from zero,
then by running the algorithm $m$ times and constructing a certain interpolation scheme, one can construct $m$ satisfying assignments such that the pairwise normalized
distance between any pair of these assignments is $\approx \beta$ w.h.p., thus contradicting Theorem~\ref{theorem:StrongClustering}.

The link between the clustering property and the ensuing demise of local algorithms was recently established by authors~\cite{gamarnik2013limits}
in a different context of finding a largest independent set in a  random regular graph. There the argument was used to show that so-called i.i.d. factor based local algorithms
are incapable of finding nearly largest independent sets in random regular graphs, refuting an earlier conjecture by Hatami, \Lovasz ~and Szegedy~\cite{HatamiLovaszSzegedy}.
The result was further strengthened by Rahman and Virag~\cite{rahman2014local}, who obtained essentially the tightest possible result, using $m$-overlap structures
of ''large'' independent set. Our use of $m$-overlaps is inspired by this work, though the set of restrictions on the $m$-overlaps implied by Theorem~\ref{theorem:StrongClustering}
is much simpler than the one appearing in~\cite{rahman2014local}.

An important technical and conceptual difference between the present work and that of~\cite{gamarnik2013limits} and \cite{rahman2014local}
is that algorithms considered in the aforementioned papers  are not sequential. Instead the decision taken by each variable in those models are
taken simultaneously for all variables.
In the case of sequential local algorithms, since the variables are set sequentially, the decision for one variable can be non-localized for the remaining variables,
this creating potential long-range dependencies. We deal with this potential long-range impact of decisions as follows. We associate variables with random i.i.d. weights chosen
from an arbitrary continuous distribution, for example a uniform distribution. The weights are used solely to determine the order of fixing the values of the variables
during the progression of the sequential local algorithm. Specifically largest weight first rule is used.
The decision to fix the value of a particular variable then can only impact variables with lower weights.
Specifically if the value of variable $x$ is fixed now, the value of variable $y$ can be impacted \emph{only} if there exists a sequence $x_0=x,x_1,\ldots,x_\ell=y$
such that the distance between $x_i$ and $x_{i+1}$ is at most $r$ (the radius of the decision making rule) and the weight of $x_i$ is larger than that of $x_{i+1}$
for all $i$.
For a given set of variables $x_0,\ldots,x_\ell$ the likelihood of this total order of variables is $1/\ell!$ which decays faster than exponential function in $l$.
This coupled with the fact that the growth rate of nodes at distance at most $r\ell$ from $x$ is at most exponential in $\ell$ (since $r$ is assumed to be constant),
will allow us to control the range of influence of the variable $x$ when its value is set. A similar idea of controlling the range of influence is used in the analysis
of local algorithms in several places, including~\cite{nguyen2008constant}. Bounding the ranges of influences is a crucial idea in implementing the interpolation scheme
and constructing $m$ assignments with "non-existence" normalized overlaps $\beta$.

\subsection{Contrast with empirical studies of SP-guided decimation}\label{subsection:SPEmpirical}
The literature on BP-, and especially, SP-guided decimation
(for instance ~\cite{braunstein2005survey}, \cite{krzakala2007gibbs}, \cite{ricci2009cavity}, \cite{MezardMontanariBook}) has shown that
these algorithms perform well empirically on random instances of $K$-SAT for small values of $K$ ($K\le 10$). There
are several choices where these implementations differ (or may differ) from
the setting we study: (1) They analyze SAT, as opposed to NAESAT; and the
asymmetry in SAT may already make a difference for the algorithm. (2) They
study 3SAT, so very local constraints, while we study $K$-clauses where $K$
is constant but large, and this increase in the
locality of the constraints may make it harder for local algorithms to function effectively (even though the locality of the algorithm
can be chosen to be arbitrarily after $K$ is fixed).
(3) In the empirically analyzed algorithms, the order in which
variables are set is not fixed a priori, but may
depend on the probability estimates returned by the message passing
iterations. While this could possibly also affect the ability of the
algorithms to finding satisfying assignments,
there appears to be no reasons based on the statistical physics theory which implies that such a presorting of variables
is a crucial for SP-guided decimation algorithm to succeed. Size biasing rather appears to be a sensible implementation detail of the algorithm. (Some discussion
of the accuracy of the size-biased version vs random order can be found in~\cite{kroc2012survey}).
(4) Finally,
and probably most significantly, we analyze algorithms that work with a
constant number of rounds of message passing iterations, and this allowed us to fit it within the framework of sequential local algorithms. In contrast the
empirical studies suggest using message passing till the iterations converge
and this may take more than linear or even exponential number of rounds.
However it is believed that the message passing procedure do converge at
some geometric rate $\gamma < 1$. Thus after $r$ iterations the ``error''
would be at most $\gamma^r$ - exponentially small in the number of
iterations.
Indeed this geometric rate of convergence is established in most examples where the model is amenable to analysis,
including~\cite{AldousSteele:survey},\cite{gamarnikMaxWeightIndSet},\cite{BandyopadhyayGamarnikCounting}, though not yet for NAE-$K$-SAT or 3SAT.
However, if this belief about the convergence rate is correct also for NAE-$K$-SAT,
then stopping after a large constant number of message passing iterations and using the estimates
to guide the sequential decision process does seem like a reasonable
heuristic.

Thus our work and setting makes a collection of choices that are different from some of the
earlier works in the hope of getting some formal analysis.
Unfortunately our result show that when the four choices are combined, it definitely produces a
provable difference, and  the algorithms fail to find satisfying assignments at densities that are qualitatively below the satisfiability threshold. Of course, it would be important to reduce the number
of parameters in which the choices for the negative results differ from
those used in the empirical setting (which yielded positive results) and we hope this will be
a subject of future work.

\subsection{Future work}

As mentioned above it would be important to understand analytically what is the largest denisty at which the SP-guided decimation
algorithm can find satisfying assignments, without the restriction of ``constant number of
message passing iterations".
In particular, it would be interesting
to investigate whether the SP-guided algorithm is capable of breaking the ''$m$-overlap barrier'' when the number of iterations is unbounded.
Conversely, it remains open to get any analytic results showing that the more complex
SP-guided decimation has an advantage over more conventional algorithms. Here it would be
interesting to see if there is any implementation choice in the algorithm that offers a
provable advantage.

Finally, going beyond specific classes of algorithms, a major challenge is to understand the
intrinsic complexity of finding satisfying assigments in random instances of $K$-SAT and
NAE-$K$-SAT problems.
Given the repeated failure to produce polynomial time algorithms for, say NAE-$K$-SAT,
above the density threshold of as $(1+o_K(1)){2^{K-1}\over K}\log^2 K$ threshold
it is plausible that the problem is actually average-case hard in this regime. The formalism of problems which are NP-hard on average is
available~\cite{levin1986average}, however
the problem which are known to be hard on average are not particularly natural and are quite distant from the types of problems considered here.
Another problem that has defied designing a fast algorithm, and which is closer in spirit to the problems considered
in this and related papers,
is the problem of finding the largest independent set in a dense random graph.
Specifically, consider the graph $\G(n,1/2)$ where every one of the $n(n-1)/2$ of the undirected edges is present with probability $1/2$ independently for all edges.
It is known that the largest independent set has size $2(1+o(1))\log_2 n$ w.h.p. At the same time the best known algorithm (greedy) finds only an independent
set of size $(1+o(1))\log_2 n$ and bridging this gap has been a major open problem in the field of combinatorics and random graphs since Karp posed this as an open
problem back in 1976~\cite{karp1976probabilistic}. It is entirely plausible that this problem is NP-hard on average, and resolving this question one way or the
other is a major open problem in theoretical computer science. It is furthermore worth noting that this problem does
indeed exhibit the clustering property at the $(1+o(1))\log_2 n$ threshold, namely at the known algorithmic threshold.
Specifically, fixing any $\beta\in (0,1)$, one can show that there exists $0<\delta_1(\beta)<\delta_2(\beta)$ such that for every pair of independent sets $I_1,I_2$ each with size
$(1+\beta+o(1))\log_2 n$ (namely the size which is existentially achievable, but not achievable by known polynomial time algorithms), it holds that
$|I_1\cap I_2|$ is either at most $(1+\delta_1+o(1))\log_2 n$ or at least $(1+\delta_2+o(1))\log_2 n$. Namely, the model exhibits the clustering property
similar to the one considered for independent sets in~\cite{coja2011independent} and~\cite{gamarnik2013limits}. It is simple to show this fact by considering
the expected values of the number of pairs of an independent sets with a given size and given overlap.
By drawing an analogy with this, admittedly
very different setup of  dense random graph, and in light of 40 years old repeated failure to produce an algorithm for this problem, it
is plausible to conjecture that NAE-K-SAT and related problems are NP-hard on average above thresholds corresponding to the emergence
of clustering type properties. Shedding some light on this question is perhaps one of the most interesting problem in the area of random constraint satisfaction problems.

\paragraph{Organization and notational conventions.}
Our main result and applications to the BP-guided and SP-guided decimation algorithms are the subject of the next section. Some preliminary technical results
are established in Section~\ref{section:LongRangeIndependence}. In particular, we establish bounds on the influence range of variables.
The property regarding $m$-overlaps of satisfying assignments is established in Section~\ref{section:Clustering}.
The proof of the main result is in Section~\ref{section:ProofOfMainTheorem}.

Throughout the paper we use standard order of magnitude notations $O(\cdot),o(\cdot)$, for sequences defined in terms of the number of Boolean variables $n$.
The constants hidden by this notation may depend on any other parameters of the model, such as $K$ and $d$. Similarly we use notations $O_K(\cdot)$ and
$o_K(\cdot)$ to denote sequences indexed by $K$ as $K\rightarrow\infty$. The constants hidden in these notations are universal.

\section{Formal statement of main result}
\label{section:NotationBasics}

In this section we formally present our main result. Before doing so
we first
introduce the mathematical notation and preliminaries needed to state
our result.

\subsection{Not-All-Equal-K-Satisfiability (NAE-$K$-SAT) problem}
At the expense of being redundant, let us recall the NAE-$K$-SAT problem.
An instance $\Phi$, of NAE-$K$-SAT problem is described as a collection of $n$ binary variables $x_1,\ldots,x_n$ taking values $0$ and $1$
and a collection of $m$ clauses $C_1,\ldots,C_m$
where each clause is given by a subset of $K$ literals. Each literal is a variable $x$ in $x_1,\ldots,x_n$ or negation $\bar x$ of a variable.
An assignment is a function $\sigma:\{x_1,\ldots,x_n\} \to \{0,1\}$.
$\sigma$ satisfies $\Phi$ if in every clause, we have at least one literal valued
$1$ and at least one literal valued $0$.
For every assignment $\sigma=(\sigma(x_i), 1\le i\le n)$, let $\bar\sigma=1-\sigma$ be the assignment given
by $\bar\sigma(x_i)=1-\sigma(x_i), 1\le i\le n$.
Given a formula $\Phi$, denote by $\mathbb{SAT}(\Phi)\subset \{0,1\}^n$
the (possibly empty) set of satisfying assignments $\sigma$.
The following ``complementation closure'' and resulting ``balance'' property
of
NAE-$K$-SAT are immediate (and do not hold for the $K$-SAT problem)

\begin{observation}\label{observation:Symmetry}
For every instance $\Phi$ of the NAE-$K$-SAT problem and assignment
$\sigma$, we have that $\sigma$ satisfies $\Phi$ if and only if
$\bar\sigma$ satisfies $\Phi$.
Consequently, suppose $\mathbb{SAT}(\Phi) \ne \emptyset$. Then if
$\sigma$ is drawn uniformly from
$\mathbb{SAT}(\Phi)$, then for every $1 \leq i \leq n$
we have $$\pr(\sigma(x_i) = 0) = \pr(\sigma(x_i) = 1) = 1/2.$$
\end{observation}

\paragraph{Reduced Instances.}

We now introduce some notations for {\em reduced} instances of
NAE-$K$-SAT. A clause of a reduced instance $C$ is given by a set
of at most $K$ literals, along with a sign $\sign(C) \in \{+,-,0\}$.
Furthermore, $C$ has exactly $K$ literals if and only if
$\sign(C) = 0$.
(Sometimes we refer to these signs as decorations.)
An assignment $\sigma$ satisfies a reduced clause $C$ if one of
the following occurs: $\sign(C) = +$ and some literal in $C$ is assigned
$0$ by $\sigma$,
OR $\sign(C) = -$ and some literal in $C$ is assigned $1$ by
$\sigma$, OR $\sign(C) = 0$ and there is at least one $0$ literal and
one $1$ literal in $C$ under the assignment $\sigma$.
A reduced NAE-$K$-SAT instance $\Phi$
consists of one or more reduced clauses, and $\sigma$ satisfies $\Phi$ if it satisfies all clauses
in $\Phi$.

Note that Observation~\ref{observation:Symmetry} does not necessarily hold for the reduced instances of NAE-$K$-SAT problem.
Instances in which every clause has sign $0$ will be called non-reduced
instances.

\paragraph{Complements }

Given a clause $C$ in a reduced instance of NAE-$K$-SAT, its
{\em complement}, denoted $\bar{C}$, is the clause with the same set
of literals, and its sign being flipped --- so
if $\sign(C) = +$ then $\sign(\bar{C}) = -$,
if $\sign(C) = -$ then $\sign(\bar{C}) = +$,
and if $\sign(C) = 0$ then $\sign(\bar{C}) = 0$.
Given a reduced instance $\Phi$ of NAE-$K$-SAT, its {\em complement}
$\bar{\Phi}$ is the instance with the complements of clauses of $\Phi$.

We now make the following observation, whose proof is immediate.

\begin{observation}\label{observation:SymmetryReduced}
Given reduced instances $\Phi$ on variables $x_1,\ldots,x_n$ and
$\Psi$ on variables $x_1,\ldots,x_{n+t}$ suppose $\Phi$ is the instance
derived by reducing $\Psi$ with the assignment
$\sigma:\{x_{n+1},\ldots,x_{n+t}\}$. Then
$\bar{\Phi}$ is the reduced instance obtained by reducing $\bar{\Psi}$
with the assignment $\bar{\sigma}$, where $\bar{\sigma}(x_i) =
1 - \sigma(x_i)$.
\end{observation}

Namely, whenever a reduced formula $\Phi$ is obtained from a non-reduced formula $\tilde\Phi$ by setting some variables of $\Psi$,
setting the same variables to opposite values generates the complement $\bar\Psi$ of $\Psi$.

\paragraph{Random NAE-$K$-SAT problem}
We denote by $\bfPhi(n,dn)$ a random (non-reduced) instance of NAE-$K$-SAT problem on variables $x_1,\ldots,x_n$ and $\lfloor dn\rfloor$ clauses $C_1,\ldots,C_m$
generated as follows. The variables in each clause $C_j$ are chosen from $x_1,\ldots,x_n$ uniformly at random without replacement,
independently for all $j=1,2,\ldots,m$. Furthermore, each $x$ variable is negated (namely appears as $\bar x$) with probability $1/2$
independently for all variables in the clause and for all clauses. We are interested in the regime when $n\rightarrow\infty$ and $d$
is constant. $d$ is called the \emph{density} of clauses to variables.

\paragraph{Graphs associated with NAE-$K$-SAT instances.}
Two graphs related to an instance $\Phi$ of
the NAE-$K$-SAT problem are important to us.
The first is the so-called {\em factor graph}, denoted $\F(\Phi)$,
which is a bipartite undirected graphs
with left nodes corresponding to the variables and right nodes corresponding to the clauses.
A clause node is connected to a variable node if and only if this variable appears in this clause.
The edges are
labelled positive or negative to indicate the polarity of the literal in
the clause.
In the cases of reduced NAE-$K$-SAT instances, clause vertices
are also labelled with the sign of the clause. Thus the factor graph of a NAE-$K$-SAT instance uniquely defines this instance.

The second graph we associate with $\Phi$ is the
{\em variable-to-variable} graph of $\Phi$, denoted
$\G(\Phi)$, which has nodes corresponding to the variables
and two nodes are adjacent if they appear in the same clause.
Note that in contrast to the factor graph, the variable-to-variable
graph loses information about the NAE-$K$-SAT instance $\Phi$.

\paragraph{Local neighborhoods}

Given a (possibly reduced) instance $\Phi$ of a NAE-$K$-SAT problem,  a variable $x$ in this instance,
and an even integer $r\ge 1$, we denote by $B_{\Phi}(x,r)$ the corresponding depth-$r$
neighborhood of $x$ in $\F(\Phi)$,
the factor graph of $\Phi$.
When the underlying formula $\Phi$ is unambiguous, we simply write $B(x,r)$.
We restrict $r$ to be even so that for every clause appearing in $B(x,r)$ all of its associated
variables also appear in $B(x,r)$.
Abusing notation slightly we also use $B(x,r)$ to denote the reduced
instance of NAE-$K$-SAT induced by the clauses in $B(x,r)$ alone.
Since $r$ is even we have that the factor graph of this induced instance
is $B(x,r)$.
In light of this, observe that $B(x,r)$ is also a reduced instance of a NAE-$K$-SAT problem.

\subsection{Sequential local algorithms for NAE-$K$-SAT problem and the main result}
We now define the notion of sequential local algorithms formally
and describe our main result.

Fix a positive even integer $r\ge 0$. Denote by $\mathcal{SAT}_r$ the set of all NAE-$K$-SAT reduced and non-reduced instances $\Psi$ with a designated
(root) variable $x$ such
that the distance from $x$ to any other variable in $\Psi$ is at most $r$ in $\F(\Psi)$. We note that $\mathcal{SAT}_r$ is an infinite set,
since even though the depth of the  factor graph $F(\Psi)$  of any $\Psi\in\mathcal{SAT}_r$ is bounded by $r$, the degree is not. The set $\mathcal{SAT}_r$
is the set of all instances $\Psi$ which can be observed as depth $r$ neighborhood $B_\Phi(x,r)$ of an arbitrary variable $x$ in an arbitrary
reduced and non-reduced  NAE-$K$-SAT instance $\Phi$.

Consider any function $\tau:\mathcal{SAT}_r\rightarrow [0,1]$ which takes as an argument an arbitrary
member $\Psi\in\mathcal{SAT}_r$ and outputs a value (probability) in
$[0,1]$.
We now describe a sequential local algorithm, which we refer to
as the {\em $\tau$-\local\ algorithm}, for
solving NAE-$K$-SAT problem.
Given a positive even integer $r$, the depth-$r$ neighborhood
$B(x_i,r)=B_{\bfPhi(n,dn)}(x_i,r)\in\mathcal{SAT}_r$ of any fixed variable $x_i\in [n]$ in the formula $\bfPhi(n,dn)$, rooted at $x_i$,
is a valid argument of the function $\tau$,  when the root
of the instance $B(x_i,r)$ is assigned to be $x_i$. This remains the case when some of the variables $x_1,\ldots,x_n$ are set to particular
values and all of the satisfied and violated clauses are removed. In this case $B(x_i,r)$ is a reduced instance. In either case,
the value $\tau(B(x_i,r))$ is well defined for every variable $x_i$ which is not set yet.
The value $\tau(B(x_i,r))$ is intended to represent the probability with which the variable $x_i$ is set to take value $1$ when its neighborhood is
a reduced or non-reduced instance $B(x_i,r)$, according to the local algorithm. Specifically, we
now describe how the function $\tau$ is used as a basis of a local algorithm to generate an assignment $\sigma:\{x_1,\ldots,x_n\}\rightarrow \{0,1\}$.

\vspace{.1in}
\noindent
{\textbf{$\tau$-\local\ algorithm}
\vspace{.1in}
\noindent\\
INPUT: \\
an instance $\Phi$ of a NAE-$K$-SAT formula on binary variables $x_1,\ldots,x_n$, \\
a positive even integer $r$,\\
function $\tau$.\\
\\
\noindent
Set $\Phi_0=\Phi$.\\
FOR $i=1:n$\\
Set $\sigma(x_i)=1$ with probability $\tau(B_{\Phi_{i-1}}(x_i,r))$ \\
Set $\sigma(x_i)=0$ with the remaining probability $1-\tau(B_{\Phi_{i-1}}(x_i,r))$.\\
Set $\Phi_i$ to be the reduced instance obtained from $\Phi_{i-1}$ by fixing the value of $x_i$ as above, removing satisfied and violated clauses and
decorating newly generated partially satisfied clauses with $+$ and $-$ appropriately. \\
\\
OUTPUT $\sigma(x_1),\ldots,\sigma(x_n)$.
\\
\vspace{.2in}\\
}
In particular, even if at some point a contradiction is reached and one of the clauses is violated, the algorithm does not stop but proceeds after
the removing violated clauses from the formula. This is assumed for convenience so that the output $\sigma(x_i)$ is well defined for all variables $x_i, 1\le i\le n$
even if the assignment turns out to be not satisfying. We denote by $\sigma_{\Phi,\tau}$ the (random) output $\sigma(x_1),\ldots,\sigma(x_n)$
produced by the $\tau$-\local\ algorithm above.
We say that $\tau$-\local\ algorithm solves instance $\Phi$ if the output $\sigma_{\Phi,\tau}$ is a satisfying assignment, namely
$\sigma_{\Phi,\tau}\in\mathbb{SAT}(\Phi)$.

We now formally define the following important symmetry condition.

\begin{definition}
\label{defn:balance}
We say that a local rule $\tau:\mathcal{SAT}_r \to [0,1]$ is {\em
balanced}
if for every instance $\Phi \in \mathcal{SAT}_r$, we have
$\tau(\bar{\Phi}) = 1 - \tau(\Phi)$.
\end{definition}

The balance condition above basically says that the
$\tau$-\local\ algorithm
does not have a prior bias in setting variables to $1$ vs $0$.
In particular, when the instance is non-reduced, $\tau$-\local\ algorithm sets variable values equi-probably, consistently with Observation~\ref{observation:Symmetry}.
This condition will allow us to take advantage of Observation~\ref{observation:SymmetryReduced} when applying the rule $\tau$ to reduced instances.

We now state the main result of the paper.
\begin{theorem}\label{theorem:MainResultNAE-$K$-SAT}
For every $\epsilon>0$ there  exists $K_0$ such that for every $K\ge K_0$,  $d>(1+\epsilon)2^{K-1}\ln^2 K/K$, every even $r>0$ and
every balanced local rule $\tau:\mathcal{SAT}_r\rightarrow [0,1]$
the following holds:
\begin{align*}
\lim_{n\rightarrow\infty}\pr(\sigma_{\bfPhi(n,dn),\tau}\in\mathbb{SAT}(\bfPhi(n,dn)))=0.
\end{align*}
\end{theorem}
Namely, with overwhelming probability, $\tau$-\local\ algorithm fails to find a satisfying assignment.
As we have mentioned above, the threshold for satisfiability is $d_s=2^{K-1}\ln 2-\ln 2/2-1/4-o_K(1)$. Thus
our theorem implies that sequential local algorithms fail to find a satisfying assignment at densities
approximately $(d_s/K)\ln^2K$.

\subsection{BP-guided and SP-guided decimation algorithms as local sequential algorithms}\label{subsection:BPandSP}

We now show that BP-guided decimation and SP-guided decimation algorithms are in fact special cases of $\tau$-\local\ algorithms as described in the
previous section, when the number of message passing iterations is bounded by a constant independent from $n$.
As a consequence we have that the negative result given by Theorem~\ref{theorem:MainResultNAE-$K$-SAT}
applies to these algorithms as well.

The BP and SP algorithms are designed to compute certain marginal values
associated with a NAE-$K$-SAT instance $\Phi$ and reduced instances
obtained after some of the variables are set.
The natural interpretation of these marginals is that variables may be set according
to these marginals sequentially, while refining the marginals as decisions are
made.
It is common to call such algorithms BP-guided decimation algorithm and
SP-guided decimation algorithms. We now describe
these algorithms in detail, starting from the BP and BP-guided decimation algorithms.

\paragraph{Belief Propagation.}

The BP algorithm is a particular message-passing type algorithm based on
variables and clauses exchanging messages on the bi-partite factor graph
$\F(\Phi(n,dn))$.
After several rounds of such exchange of messages, the messages are combined
in a specific way to compute marginal probabilities.

However, the relevant part for us is the fact that if the messages are passed only a constant $r$ number of rounds, then for every variable $x_i$
such that the neighborhood $B(x_i,r)$ is in fact a tree,  the computed marginals $\mu(x_i)$ are precisely the ratio of the number
of assignments satisfying NAE-$K$-SAT formula $B(x_i,r)$ which set $x_i$ to one to the number of such assignments which set this variable to zero.
A standard fact is that for the majority of variables  $B(x_i,r)$ is indeed a tree.
Thus most of the times BP iterations compute marginal values corresponding to the ratio described above.
These marginals are then used to design the BP-guided decimation algorithm as follows. Variable $x_1$ is selected and BP algorithm is used to compute
its marginal $\mu(x_1)$ with respect to the neighborhood tree $B(x_1,r)$. Then the decision $\sigma(x_1)$ for this variable is set to $\sigma(x_1)=1$ with probability
$\mu(x_1)/(\mu(x_1)+1)$ and $\sigma(x_1)=0$ with probability $1/(\mu(x_1)+1)$. Namely, the variable is set probabilistically proportionally to the ratio
of the number of solutions setting it to one vs the number of solutions setting it to zero. After the decision for variable $x_1$ is set in the way described above,
the variable $x_2$ is selected  from  the reduced formula on variables $x_2,\ldots,x_n$.
The marginal $\mu(x_2)$ with respect to the neighborhood $B(x_2,r)$ for this reduced formula is computed and the value $\sigma(x_2)$ is determined based on
$\mu(x_2)$ similarly, and so on. The procedure is called BP-guided decimation algorithm. It is thus parametrized by the computation depth $r$.

It is clear that the such a BP-guided decimation algorithm is precisely the $\tau$-\local\ algorithm
where $\tau(B(x_i,r))=\mu(x_i)/(\mu(x_i)+1)$ - the marginal probability of the variable
corresponding to the reduced formula $B(x_i,r)$.
Furthermore, such $\tau$ rule satisfies the balance condition described in Definition~\ref{defn:balance}.
Thus, as an implication of our main result, Theorem~\ref{theorem:MainResultNAE-$K$-SAT}, we conclude that BP-guided
decimation algorithm fails to find a satisfying assignment for $\Phi(n,dn)$ in the regime where our result on $\tau$-\local\ algorithms applies:
\begin{coro}\label{coro:BP}
There exists $K_0$ such that for every $K\ge K_0$,  $d>2^{K-2}\ln 2$ and $r>0$
\begin{align*}
\lim_{n\rightarrow\infty}\pr(\text{BP-guided decimation algorithm with $r$ iterations solves}~\bfPhi(n,dn))=0.
\end{align*}
\end{coro}

\paragraph{Survey Propagation.}
We now describe Survey Propagation-guided decimation algorithm in similar level of detail.
The algorithm is significantly more complex to describe, but we will show
again that it is a $\tau$-\local\ algorithm when the number of message passing rounds is bounded by a constant independent from $n$, and that
$\tau$ is a balanced rule. As a consequence we will conclude that SP-guided decimation algorithm
also fails to find satisfying assignments for instances with density
larger than $(d_s/K)\ln^2K$ when the number of rounds is bounded by a constant. This is summarized in Corollary~\ref{coro:SP} below.

The setup is similar to the one for BP. In particular in steps $i=1,2,\ldots,n$
certain marginal value is computed and the decision for $x_i$ is again based on this marginal value, except now the marginal values do not correspond
to the ratio of the number of assignments, but rather correspond to ratios when the problem is lifted to a new certain constraint satisfaction problem
with decision variables $0,1,*$. We do not describe here the rationale for this lifting procedure, as this has been documented in many
papers, including~\cite{braunstein2005survey},\cite{maneva2007new},\cite{mezard2002analytic},\cite{MezardMontanariBook}. Instead we simply formally present the SP algorithm
and SP-guided decimation algorithm, following closely \cite{MezardMontanariBook} with the appropriate adjustment from the K-SAT
problem to the  NAE-$K$-SAT problem. We will convince ourselves that SP-guided decimation algorithm is again the special case of a balanced $\tau$-\local\ algorithm.
We will then be able to conclude that SP-guided decimation algorithm fails to find a satisfying assignment with probability approaching unity, in the regime outlined
in our main result, Theorem~\ref{theorem:MainResultNAE-$K$-SAT}.

The SP algorithm is an iterative scheme described as follows. The details and notations are very similar to the ones described in~\cite{MezardMontanariBook}.
Specifically iterations (\ref{eq:barQ1})-(\ref{eq:Q3}) below correspond to iterations (20.17)-(20.20) in this book.
Consider an arbitrary reduced or non-reduced NAE-$K$-SAT formula $\Phi$ on variables $x_1,\ldots,x_N$.
For each iteration $t=0,1,\ldots$, each variable/clause pair $(x,C)$ such that $x$ appears in $C$ (namely there is an edge between $x$ and $C$ in the bi-partite factor
graph representation) is associated
with five random variables $Q_{x,C,U}^t,Q_{x,C,S}^t,Q_{x,C,*}^t,Q_{C,x,S}^t$ and $Q_{C,x,U}^t$. Here is the interpretation of these variables.
Each of them is a message send from variable to a clause containing this variable, or a message from a clause to a variable which belongs to this clause. Specifically,
$Q_{x,C,U}^t(Q_{x,C,S}^t)$ is interpreted as the probability computed at iteration $t$
that the variable $x$ is forced by clauses $D$ other than $C$ to take value which does not (does) satisfy $C$.
$Q_{x,C,*}^t$ is interpreted that none of these forcing takes place. $Q_{C,x,S}^t$ is interpreted as probability computed at iteration $t$ that
all variables $y\in C$ other than $x$ do not satisfy $C$, and thus the only hope of satisfying $C$ is for $x$ to do so.
Similarly, $Q_{C,x,U}^t$ is the probability that all variables $y$ in $C$ other than $x$ do satisfy $C$ and thus the only hope of satisfying
clause $C$ is for $x$ to violate it. The latter case is an artifact of the NAE variant of the problem and need not be introduced in the SP iterations for the
K-SAT problem.

The variables $Q^t$  are then computed as follows.
At time $t=0$ the variables are generated uniformly at random from $[0,1]$ independently for all five variables. Then they are normalized
so that $Q_{x,C,U}^0+Q_{x,C,S}^0+Q_{x,C,*}^0=1$, which is achieved by dividing each term by the sum $Q_{x,C,U}^0+Q_{x,C,S}^0+Q_{x,C,*}^0$.
Similarly, variables $Q_{C,x,S}^0$ and $Q_{C,x,U}^0$ are normalized to sum to one.

Now we describe the iteration procedures at times $t\ge 0$.
For each such pair $x,C$, let $\mathcal{S}_{x,C}$ be the set of clauses containing $x$ other than
$C$, in which $x$ appears in the same way as in $C$. Namely if $x$ appears in $C$ without negation, it appears without negation in clauses in $\mathcal{S}_{x,C}$ as well.
Similarly, if $x$ appears as $\bar x$ in $C$, the same is true for clauses in $\mathcal{S}_{x,C}$. Let $\mathcal{U}_{x,C}$ be the remaining set of clauses containing $x$, namely
clauses, where $x$ appears opposite to the way it appears in $C$. Now for each $t=0,1,2,\ldots$ assume $Q_{x,C,U}^t,Q_{x,C,S}^t,Q_{x,C,*}^t,Q_{C,x,S}^t$ and $Q_{C,x,U}^t$
are defined.  Define the random variable $Q_{x,C,S}^{t+1}$ and $Q_{x,C,U}^{t+1}$ as follows. Suppose $C$ is unsigned in $\Phi$. Then
\begin{align}\label{eq:barQ1}
Q_{C,x,S}^{t+1}=\prod_{y\in C\setminus x}Q_{y,C,U}^t,
\end{align}
and
\begin{align}\label{eq:barQ2}
Q_{C,x,U}^{t+1}=\prod_{y\in C\setminus x}Q_{x,C,S}^t.
\end{align}
Here $C\setminus x$ is the set of variables in clause $C$ other than $x$. The interpretation for this identities is as follows.
When $C$ is not signed, the clause $C$ forces its variable $x$ to satisfy it if all other variables $y$ in $C$ where forced not to satisfy $C$ at previous iteration
due to other clauses. The first identity is the probability of this event assuming the events ''$y$ is forced not to satisfy $C$'' are independent.
The second identity is interpreted similarly, though it is only relevant only for NAE-$K$-SAT problem and does not appear for the corresponding iterations for the K-SAT problem.

If the clause $C$ is signed $+$, then we set $Q_{C,x,S}^{t+1}=0$ and
\begin{align}\label{eq:Q1}
 Q_{C,x,U}^{t+1}=\prod_{y\in C\setminus x}Q_{x,C,S}^t.
\end{align}
The interpretation is that if $C$ is signed $+$, then one of the variables was already set to satisfy it. Thus the only way
the clause $C$ can force $x$ to violate it is when all other variables $y$ are forced to satisfy $C$. Again this is only relevant for the NAE-$K$-SAT problem.
Similarly, if $C$ is signed $-$, then $ Q_{C,x,U}^{t+1}=0$ and
\begin{align}\label{eq:Q2}
 Q_{C,x,S}^{t+1}=\prod_{y\in C\setminus x}Q_{x,C,U}^t.
\end{align}
Next we define variables $R_{x,C,S}^{t+1},R_{x,C,U}^{t+1}$ and $R_{x,C,*}^{t+1}$ which stand for $Q_{x,C,S}^{t+1},Q_{x,C,U}^{t+1}$ and $Q_{x,C,*}^{t+1}$
before the normalization. These random variables are computed  using the following rules:
\begin{align}\label{eq:Q3}
R_{x,C,S}^{t+1}&=\prod_{D\in\mathcal{U}_{x,C}}(1- Q_{D,x,S}^t)
\prod_{D\in\mathcal{S}_{x,C}}(1- Q_{D,x,U}^t)
-\prod_{D\in\mathcal{U}_{x,C}}(1- Q_{D,x,*}^t)
\prod_{D\in\mathcal{S}_{x,C}}(1- Q_{D,x,*}^t),
\end{align}
which is interpreted as follows. The first term in the right-hand side of the expression above is interpreted as the probability
that none of the clauses $D$ in $\mathcal{U}_{x,C}$ force $x$ to take value which satisfies $D$ and therefore violates $C$ (since otherwise a contradiction would be reached)
and none of the clauses $D$ in $\mathcal{S}_{x,C}$ force $x$ to take value which violates $D$ and therefore violates $C$ (since otherwise a contradiction would be reached).
The second term term in the right-hand side is interpreted as the probability variable $x$ is not forced to take any particular value by clauses it belongs to other than $C$.
The difference of the two terms is precisely the probability that $x$ is forced to take value satisfying $C$ and is not forced to take value contradicting this choice.

Similarly, define
\begin{align}\label{eq:Q4}
R_{x,C,U}^{t+1}&=
\prod_{D\in\mathcal{U}_{x,C}}(1- Q_{D,x,U}^t)
\prod_{D\in\mathcal{S}_{x,C}}(1- Q_{D,x,S}^t)
-\prod_{D\in\mathcal{U}_{x,C}}(1- Q_{D,x,*}^t)
\prod_{D\in\mathcal{S}_{x,C}}(1- Q_{D,x,*}^t).
\end{align}
The interpretation for $R_{x,C,U}^{t+1}$ is similar: it is the probability that $x$ is forced to take value violating $C$ and is not forced a contradicting value of satisfying $C$.
Next, define
\begin{align}\label{eq:Q5}
R_{x,C,*}^{t+1}=\prod_{D\in\mathcal{S}_{x,C}\cup \mathcal{U}_{x,C}}(1- Q_{D,x,S}^t- Q_{D,x,U}^t).
\end{align}
$R_{x,C,*}^{t+1}$ is interpreted as the probability that $x$ is not forced in either way by clauses other than $C$. Finally, we
let $Q_{x,C,S}^{t+1},Q_{x,C,U}^{t+1}$ and $Q_{x,C,*}^{t+1}$ to be quantities $R_{x,C,S}^{t+1},R_{x,C,U}^{t+1}$ and $R_{x,C,*}^{t+1}$, respectively,
normalized by their sum $R_{x,C,S}^{t+1}+R_{x,C,U}^{t+1}+R_{x,C,*}^{t+1}$, so that the three variables sum up to one.
The iterations (\ref{eq:barQ1})-(\ref{eq:Q3})
are conducted for some number of steps $t=0,1,\ldots,r$. Next variables $W_x(1)$ and $W_x(0)$ and $W_x(*)$ are computed for all variables $x$ as follows.
Let $\mathcal{S}_x$ be the set of clauses where $x$ appears without negation and let $\mathcal{U}_x$ be the set of clauses where $x$ appears with negation.
Then set
\begin{align}\label{eq:Wx1}
W_x(1)&=\prod_{D\in\mathcal{U}_{x}}(1- Q_{D,x,S}^t)
\prod_{D\in\mathcal{S}_{x}}(1- Q_{D,x,U}^t)
-\prod_{D\in\mathcal{U}_{x}}(1- Q_{D,x,*}^t)
\prod_{D\in\mathcal{S}_{x}}(1- Q_{D,x,*}^t).
\end{align}
$W_x(1)$ is interpreted as probability (after normalization) that variable $x$ is forced to take value $1$, but is not forced to take value zero by all of the clauses
containing $x$. Similarly, we set
\begin{align}\label{eq:Wx0}
W_x(0)&=\prod_{D\in\mathcal{S}_{x}}(1- Q_{D,x,S}^t)
\prod_{D\in\mathcal{U}_{x}}(1- Q_{D,x,U}^t)
-\prod_{D\in\mathcal{S}_{x}}(1- Q_{D,x,*}^t)
\prod_{D\in\mathcal{U}_{x}}(1- Q_{D,x,*}^t).
\end{align}
with a similar interpretation. Then set
\begin{align}\label{eq:Wx*}
W_x(*)&=\prod_{D\in\mathcal{S}_{x}\cup \mathcal{U}_{x}}(1- Q_{D,x,S}^r- Q_{D,x,U}^r),
\end{align}
which is interpreted as the probability (after normalization) that $x$ is not take forced to be either $0$ or $1$. Finally, the values
$W_x(0),W_x(1),W_x(*)$ are normalized to sum up to one. For simplicity we use the same notation for these quantities after normalization.

The random variables $W_x(0),W_x(1),W_x(*)$ are used to guide the decimation algorithm as follows. Given a random formula $\bfPhi(n,dn)$,
variable $x_1$ is selected. The random quantities $W_{x_1}(0),W_{x_1}(1)$ and $ W_{x_1}(*)$ are computed and $x_1$ is set to $1$
if $W_{x_1}(1)>W_{x_1}(0)$ and set it to zero otherwise. The formula is now reduced and contains variables $x_2,x_3,\ldots,x_n$.
Variable $x_2$ is then selected and the random quantities $W_{x_2}(0),W_{x_2}(1)$ are computed with respect to the reduced formula.
Then $ W_{x_2}$ is computed and  $x_2$ is set to $1$ if $W_{x_2}(1)>W_{x_2}(0)$, and set to zero otherwise. The procedure is repeated until
all variables are set. This defines the SP-guided decimation algorithm.

It is clear again that the SP-guided decimation algorithm is the special case of $\tau$-\local\ algorithm,
where $\tau$ function corresponds to the probability of the event $W_x(1)>W_x(0)$, when it applies to a reduced instance $B(x,r)$ with $x$ as its root. The depth
$r$ of the instance corresponds to the number of iterations of the SP procedure. Furthermore, we claim that this $\tau$ rule is balanced.

\begin{prop}\label{prop:tauSPbalanced}
The local rule $\tau$ corresponding to the Survey Propagation iterations is balanced.
\end{prop}

\begin{proof}
Recall that at the iteration $t=0$, the variables $Q^t$  are chosen independently uniformly at random from $[0,1]$, normalized appropriately.
The main idea of the proof is to use the symmetry of the uniform distribution. Given a formula $\Phi$, we claim
that if we initialize random variables $Q^r$ with
variables $Q_{x,C,U}^0$ and $Q_{x,C,S}^0$ swapped, variables $Q_{C,x,S}^0$ and $ Q_{C,x,U}^0$ swapped, variables $Q_{x,C,*}^0$ left intact,
and apply it to formula $\bar\Phi$ instead of $\Phi$,
we obtain values $W_x(0),W_x(1)$ and $W_x(*)$ such that under this initialization  $W_x(1)>W_x(0)$ holds iff $W_x(0)<W_x(1)$
under the original initialization for the original formula $\Phi$. The claim of the proposition then follows.

We now establish the claim by a simple inductive reasoning. As suggested above, given $Q_{x,C,U}^0$, $Q_{x,C,S}^0$, $Q_{x,C,*}^0$, $ Q_{C,x,S}^0$ and $ Q_{C,x,U}^0$
(after normalization for concreteness), define
\begin{align}
P_{x,C,U}^0&=Q_{x,C,S}^0, \notag\\
P_{x,C,S}^0&=Q_{x,C,U}^0, \notag\\
P_{x,C,*}^0&=Q_{x,C,*}^0, \notag\\
 P_{C,x,S}^0&= Q_{C,x,U}^0, \notag\\
 P_{C,x,U}^0&= Q_{C,x,S}^0. \label{eq:PQsymmetry}
\end{align}
Then define variables $P_{x,C,U}^t,P_{x,C,S}^t,P_{x,C,*}^t, P_{C,x,S}^t$ and $ P_{C,x,U}^t$ with respect to the formula $\bar\Phi$ similarly to the way
variables $Q_{x,C,U}^t,Q_{x,C,S}^t,Q_{x,C,*}^t, Q_{C,x,S}^t$ and $ Q_{C,x,U}^t$ are defined with respect to the formula $\Phi$. We now prove by induction that
the identities (\ref{eq:PQsymmetry}) hold for general $t$ and not just when $t=0$. The base of the induction is given by (\ref{eq:PQsymmetry}). Assume the claim holds for
$t'\le t-1$. Consider any unsigned clause $C$ in $\bar\Phi$. Then this clause is unsigned in $\Phi$ as well.
Applying (\ref{eq:barQ1}) and (\ref{eq:barQ2}), and the inductive assumption, we conclude that the claim holds for $ P_{C,x,S}^t$ and $ P_{C,x,U}^t$ as well. Similarly, if a clause $C$ is signed $+$ in $\bar\Phi$,
then it is signed $-$ in $\Phi$. Applying identities (\ref{eq:Q1}) and (\ref{eq:Q2}), the claim holds for $ P_{C,x,S}^t$ and $ P_{C,x,U}^t$ as well.
The case when $C$ is signed $-$ in $\bar\Phi$ is considered similarly.

We now establish the claim for the three remaining variables $P_{x,C,S}^t,P_{x,C,U}^t,P_{x,C,*}^t$. Note that the sets of clauses $\mathcal{S}_{x,C}$ and
$\mathcal{U}_{x,C}$ are the same for the formulas $\Phi$ and $\bar\Phi$. Applying (\ref{eq:Q3}) to compute $P_{x,C,U}^t$, using the inductive
assumption $ P_{C,x,S}^{t-1}= Q_{C,x,U}^{t-1}, P_{C,x,U}^{t-1}= Q_{C,x,S}^{t-1}$, and comparing with (\ref{eq:Q4}), we see
that $P_{x,C,S}^t=Q_{x,C,U}^t$. Similarly, we see that $P_{x,C,U}^t=Q_{x,C,S}^t$. Finally, applying (\ref{eq:Q5}), we see that $P_{x,C,*}^t=Q_{x,C,*}^t$.
This completes the proof of the induction.

Now define $Z_x(0),Z_x(1)$ and $Z_x(*)$ in terms of $ P^r$ in the same way as $W_x(0),W_x(1)$ and $W_x(*)$ are defined in terms of $ Q^r$, namely
via identities (\ref{eq:Wx1}),(\ref{eq:Wx0}) and (\ref{eq:Wx*}). Again we see that $Z_x(0)=W_x(1),Z_x(1)=W_x(0)$ and $Z_x(*)=W_x(*)$,
further implying $\pr(Z_x(1)>Z_x(0))=1-\pr(W_x(1)>W_x(0))$. Thus the rule $\tau(B_\Phi(x,r))=\pr(W_x(1)>W_x(0))$ is balanced.
\end{proof}

Theorem~\ref{theorem:MainResultNAE-$K$-SAT} then becomes applicable and we conclude:
\begin{coro}\label{coro:SP}
There exists $K_0$ such that for every $K\ge K_0$,  $d>2^{K-2}\ln 2$ and $r>0$
\begin{align*}
\lim_{n\rightarrow\infty}\pr(\text{SP-guided decimation algorithm with $r$ iterations solves}~\bfPhi(n,dn))=0.
\end{align*}
\end{coro}

\section{Local algorithms and long-range independence}\label{section:LongRangeIndependence}
In this section we obtain some preliminary results needed for the proof of our main result, Theorem~\ref{theorem:MainResultNAE-$K$-SAT}.
Specifically we prove two structural results about the $\tau$-\local\ algorithm for a local rule $\tau$.

The first result is simple to state - we show that balanced local rules
lead to unbiased decisions for {\em every} non-reduced NAE-$K$-SAT
instance: specifically the marginal probability that a variable
is set to $1$ is $1/2$. More generally we show that the probability
that a variable is set to $1$ in any reduced or non-reduced instance
$\Phi$ equals the probability that the same
variable is set to $0$ in the complementary instance $\bar\Phi$.
(See Lemma~\ref{lemma:equiprob}.) This lemma later allows us to find satisfying
assignments
with small overlap in random instances $\Phi(n,dn)$.

Next,
we consider the ``influence'' of a decision $\sigma(x_i)
\in \{0,1\}$ and ask how many other variables are affected by this
decision.
In particular, we show that the decisions $\sigma$ assigned
to a pair of fixed variables $x_i$ and $x_j$ are asymptotically
independent as $n\rightarrow\infty$. Namely, the decisions exhibit
a long-range independence. Such a long-range independence
is not a priori obvious, since setting a value of a variable $x_i$
can have a downstream implications for setting variables $x_j, j\ge i$.
We will show, however, that the chain of implications, appropriately
defined is typically
short. Definition~\ref{defn:IR}
and Proposition~\ref{prop:MaximumIR} formalize these claims.

In what follows, we first introduce some notation that makes
the decisions of the randomized algorithm more formal and precise.
We then prove the two main claims above in the following subsections.

\subsection{Formalizing random choices of a $\tau$-\local\ algorithm}

The $\tau$-\local\ algorithm described in the previous section
is based on the ordering of the variables $x_i$, since the values
$\sigma(x_i)$ are set in the order $i=1,2,\ldots,n$.
In the case of the random NAE-$K$-SAT formula $\bfPhi(n,dn)$,
due to symmetry we may assume, without the loss of generality, that
the ordering is achieved by assigning random i.i.d. labels chosen uniformly from $[0,1]$
and using order statistics for ordering of variables. (This is equivalent
to renaming the variables at random and this renaming will be
convenient for us.)
Specifically, let $\bfZ=(Z_i, 1\le i\le n)$ be the i.i.d. sequence of random variables with uniform in $[0,1]$
distribution. Let $\pi:[n]\rightarrow [n]$ be the permutation induced by the order statistics of $\bfZ$. Namely $Z_{\pi(1)}>Z_{\pi(2)}>\cdots> Z_{\pi(n)}$.
We now assume that when the $\tau$-\local\ algorithm is performed, the first variable selected is $x_{\pi(1)}$ (as opposed to $x_1$), the second variable selected
is $x_{\pi(2)}$ (as opposed to $x_2$), etc. Namely, we assume that $\tau$-\local\ algorithm performed on a random instance of the NAE-$K$-SAT problem
$\bfPhi(n,dn)$ is conducted according to this ordering.

To facilitate the randomization involved in selecting randomized decisions based on the $\tau$ rule, consider
another i.i.d. sequence $\bfU=(U_i, 1\le i\le n)$ of random variables with the uniform in $[0,1]$ distribution,
which is independent from the randomness of the instance $\bfPhi$ and
sequence $\bfZ$. The purpose of the sequence is to serve as random seeds for the decision $\sigma(x_i)$ based on $\tau$. Specifically, when the value $\sigma(x_i)$ associated
with variable $x_i$ is determined, it is done so according to the rule $\sigma(x_i)=1$ if $U_i<\tau(B(x_i,r))$ and $\sigma(x_i)=0$ otherwise,
where $B(x_i,r)=B_{\Phi_{i-1}}(x_i,r)$
is the reduced NAE-$K$-SAT instance rooted at $x_i$, observed at a time when the decision for $x_i$ needs to be made. Namely, the $\tau$-\local\ algorithm is faithfully executed.
Conditioned on $\bfZ,\bfU$ and $\bfPhi$, the output $\sigma:[n]\rightarrow \{0,1\}$ is uniquely determined. We denote by $\sigma_{\Phi,\bfz,\bfu}(x_i), 1\le i\le n$
the output of the $\tau$-\local\ algorithm conditioned on the realizations $\Phi,\bfz,\bfu$ of the random instance $\bfPhi(n,dn)$, vector $\bfZ$ and vector $\bfU$, respectively.
Similarly, we denote by $B_{\Phi,\bfz,\bfu}(x_i,r), 1\le i\le n$ the (possibly) reduced NAE-$K$-SAT instance corresponding to the $r$-depth neighborhood of variable $x_i$
at the time when the value of $x_i$ is determined by the $\tau$-\local\ algorithm. In particular, $\sigma_{\Phi,\bfz,\bfu}(x_i)=1$ if
$u_i\in [0,\tau(B_{\Phi,\bfz,\bfu}(x_i,r))]$ and $\sigma_{\Phi,\bfz,\bfu}(x_i)=0$ if $u_i\in (\tau(B_{\Phi,\bfz,\bfu}(x_i,r)),1]$.

\subsection{Implications of balance}

We now establish the following implication of the the Definition~\ref{defn:balance} of balanced local rules.
\begin{lemma}\label{lemma:equiprob}
For every formula $\Phi$,  and vectors $\bfz,\bfu$, the following identities hold for every variable $x_i, 1\le i\le n$:
\begin{align}
B_{\Phi,\bfz,\bar\bfu}(x_i,r)&=\bar B_{\Phi,\bfz,\bfu}(x_i,r), \label{eq:B=barB}\\
\sigma_{\Phi,\bfz,\bar\bfu}(x_i)&=1-\sigma_{\Phi,\bfz,\bfu}(x_i)\label{sigma=1-sigma},
\end{align}
where $\bar\bfu$ is defined by $\bar u_i=1-u_i, 1\le i\le n$. As a result, when $\bfU$ is a vector of i.i.d. random variables chosen uniformly from $[0,1]$,
for $\Phi$ and $\bfz$, the following holds for all $i=1,2,\ldots,n$:
\begin{align}\label{eq:Probhalf}
\pr(\sigma_{\Phi,\bfz,\bfU}(x_i)=0)=1/2.
\end{align}
\end{lemma}
Note, that the randomness in the probability above is with respect to $\bfU$ only and the claim holds for every formula $\Phi$ and every vector $\bfz$.
\begin{proof}
We prove the claim by induction on $x_{\pi(1)},x_{\pi(2)},\ldots,x_{\pi(n)}$, where $\pi$ is the permutation generated by $\bfz$,
that is $z_{\pi(1)}>z_{\pi(2)}>\cdots>z_{\pi(n)}$.
Specifically, we will show by induction that for every $i=0,1,2,\ldots,n$, just before the value of variable $x_{\pi(i)}$ is determined,
the identity (\ref{sigma=1-sigma}) holds for all variables $x_{\pi(j)}, j\le i-1$ (namely for variables for which the value is already determined at time $i$),
and the identity (\ref{eq:B=barB}) in fact holds for all neighborhoods $B_{\Phi,\bfz,\bfu}(x_{\pi(k)},r), i\le k\le n$ and $B_{\Phi,\bfz,\bar\bfu}(x_{\pi(k)},r), i\le k\le n$,
and not just for $B_{\Phi,\bfz,\bfu}(x_{\pi(i)},r)$ and $B_{\Phi,\bfz,\bar\bfu}(x_{\pi(i)},r)$.

For the base of the induction corresponding to $i=1$, no variables are set yet and all the neighborhoods
$B_{\Phi,\bfz,\bfu}(x_k,r),B_{\Phi,\bfz,\bar\bfu}(x_k,r), 1\le k\le n$ correspond to non-reduced instances, for which by our convention, its symmetric complement
is the instance itself. Namely $B_{\Phi,\bfz,\bar\bfu}(x_k,r)=\bar B_{\Phi,\bfz,\bar\bfu}(x_k,r)=B_{\Phi,\bfz,\bfu}(x_k,r)$, and thus (\ref{eq:B=barB}) is verified.

Fix $i\ge 1$ and assume now the inductive hypothesis holds for $j\le i$. In particular, the values $\sigma(x_{\pi(j)})$ are determined for $j=1,\ldots,i-1$
under $\bfu$ and $\bar\bfu$. Now consider the step of assigning the value of $x_{\pi(i)}$. We have $\sigma_{\Phi,\bfz,\bfu}(x_{\pi(i)})=1$
iff $u_{\pi(i)}\le \tau(B_{\Phi,\bfz,\bfu}(x_{\pi(i)},r))$ and $\sigma_{\Phi,\bfz,\bar\bfu}(x_{\pi(i)})=1$
iff $\bar u_{\pi(i)}\le \tau(B_{\Phi,\bfz,\bar\bfu}(x_{\pi(i)},r))$.  By the inductive assumption we have that
$B_{\Phi,\bfz,\bar\bfu}(x_{\pi(i)},r)=\bar B_{\Phi,\bfz,\bfu}(x_{\pi(i)},r)$. Since $\tau$ is balanced, we have
$\tau(\bar B_{\Phi,\bfz,\bfu}(x_{\pi(i)},r))=1-\tau(B_{\Phi,\bfz,\bfu}(x_{\pi(i)},r))$. Since $\bar u=1-u$, we conclude
that $\sigma_{\Phi,\bfz,\bfu}(x_{\pi(i)})=1$ iff $\sigma_{\Phi,\bfz,\bar\bfu}(x_{\pi(i)})=0$ and vice verse. Namely,
$\sigma_{\Phi,\bfz,\bfu}(x_{\pi(i)})=1-\sigma_{\Phi,\bfz,\bar\bfu}(x_{\pi(i)})$ and identity (\ref{sigma=1-sigma}) is verified.

It remains to show that identity (\ref{eq:B=barB}) still holds for all variables after the value $\sigma(x_{\pi(i)})$ is determined. All neighborhoods
$B(x_k,r)$ which do not contain $x_{\pi(i)}$ are not affected by fixing the value of $x_{\pi(i)}$ and thus the identity holds by the inductive assumption.
Suppose $B(x_k,r)$ contains $x_{\pi(i)}$. This means this neighborhood contains one or several clauses  which contains $x_{\pi(i)}$.
Fix any such clause $C$. If this clause was unsigned under $\bfu$, then by the inductive assumption it was also unsigned under $\bar\bfu$ (as the instances
under $\bfu$ and $\bar\bfu$ are complements of each other).
The clause then becomes signed after fixing the value of $x_{\pi(i)}$, and,  furthermore,  the signs will be opposite under $\bfu$ and $\bar\bfu$,
since (\ref{sigma=1-sigma}) holds for $x_{\pi(i)}$ as we have just established.

Now suppose the clause was signed $+$ under $\bfu$. Then again by the inductive
assumption it was signed $-$ under $\bar\bfu$. In this case if the assignment $\sigma_{\Phi,\bfz,\bfu}(x_{\pi(i)})$ satisfies $C$, then the clause
remains signed $+$ after setting the value of $x_{\pi(i)}$.
At the same time this means that $\sigma_{\Phi,\bfz,\bar\bfu}(x_{\pi(i)})=1-\sigma_{\Phi,\bfz,\bfu}(x_{\pi(i)})$ does not satisfy
$C$ and the clause remains signed $-$ after setting the value of $x_{\pi(i)}$. In both cases the variable $x_{\pi(i)}$ is deleted and
the identity (\ref{eq:B=barB}) still holds.
On the other hand if $\sigma_{\Phi,\bfz,\bfu}(x_{\pi(i)})$ does not satisfy $C$
when $\bfu$ is used, then (since it was signed $+$) the clause $C$ is now satisfied and disappears from the formula. But at the same time
this means $\sigma_{\Phi,\bfz,\bar\bfu}(x_{\pi(i)})$  satisfies $C$, since it was signed $-$ under $\bar\bfu$, and therefore $C$ is satisfied again
and disappears from the formula. The variable $x_{\pi(i)}$ is deleted in both cases and again (\ref{eq:B=barB}) is verified.

The case when clause $C$ is signed $-$ under $\bfu$ and signed $+$ under $\bar\bfu$ is considered similarly. Finally, suppose $\sigma_{\Phi,\bfz,\bfu}(x_{\pi(i)})$
violates a clause $C$ containing $x_{\pi(i)}$. This means that $C$ contains only this variable when setting this variable to $\sigma_{\Phi,\bfz,\bfu}(x_{\pi(i)})$.
By inductive assumption we see that the same is true under $\bar\bfu$. In both cases both the variable and clause are removed from the formula.
This completes the proof of the inductive step.

Finally, since the distribution of $\bfU$ and $\bar\bfU$ is identical for i.i.d. sequences chosen uniformly at random from $[0,1]$, we obtain (\ref{eq:Probhalf}).
\end{proof}

\subsection{Influence ranges}

We now define the notion of influence (which depends on the
formula $\bfPhi(n,dn)$ and ordering $\bfZ$, but not on random
choices of the $\tau$-\local\ algorithm).
We introduce the following relationship between the variables
$x_1,\ldots,x_n$ of our formula.

\begin{definition}
\label{defn:IR}
Given a random formula $\bfPhi(n,dn)$ and random sequence $\bfZ$
we say that $x_i$ \emph{influences} $x_j$ if either $x_j=x_i$ or in the underlying node-to-node graph $\G=\G(\bfPhi(n,dn))$ there exists
a sequence of nodes $y_0,y_1,\ldots,y_t\in \{x_1,\ldots,x_n\}$ with the following properties:
\begin{enumerate}
\item[(i)] $y_0=x_i$ and $y_t=x_j$.
\item[(ii)] $y_l$ and $y_{l+1}$ are connected by a path of length at most $r$ in graph $\G$ for all $l=0,1,\ldots,t-1$.
\item[(iii)] $Z_{y_{l-1}}>Z_{y_l}$ for $l=1,2,\ldots,t$. In particular, $Z_{x_i}>Z_{x_j}$.
\end{enumerate}
In this case we write $x_i\rightsquigarrow x_j$. We denote by $\mathcal{IR}_{x_i}$ the set of variables $x_j$ influenced by $x_i$
and call it \emph{influence range} of $x_i$.
\end{definition}

Note that indeed the randomness underlying the sets $\mathcal{IR}_{x_i}, 1\le i\le n$ as well as the relationship $\rightsquigarrow$ is the function of the
randomness of the formula $\bfPhi(n,dn)$ and vector $\bfZ$, but is independent from the random vector $\bfU$.

While the definition above is sound for every constant $r>0$, we will apply it in the case where $r$ is the parameter appearing in the context of $\tau$-\local\ algorithm.
Namely, in the context of the $\tau$ function  defined the set of rooted instances $\mathcal{SAT}_r$ introduced above.
In this case the notion of influence range is justified by the following observation.
\begin{prop}\label{prop:InfluenceResistance}
Given realizations $\Phi$ and $\bfz$ of the random formula $\bfPhi(n,dn)$ and random ordering $\bfZ$,
suppose $\bfu=(u_i,1\le i\le n)$ and $\bfu'=(u_i', 1\le i\le n)$ are such that $u_{i_0}=u'_{i_0}$ and $u_i=u_i'$ for all $i\ne i_0$, for some fixed index $i_0$.
Then $\sigma_{\Phi,\bfz,\bfu}(x)=\sigma_{\Phi,\bfz,\bfu'}(x)$ for every $x\notin\mathcal{IR}_{i_0}$.
That is, changing the values of $\bfu$ at $i_0$ may impact the decisions associated with only variables $x$ in $\mathcal{IR}_{x_{i_0}}$
\end{prop}

\begin{proof}
Given a variable $x_i, i\ne i_0$, in order for its decision $\sigma_{\Phi,\bfz,\cdot}(x_i)$ to be affected by switching from $\bfu$ to $\bfu'$,
there must exist a variable $x_{i_1}$ with distance at most $r$ (with respect to the node-to-node graph $\G=\G(\Phi)$) from $x_i$
such that $z_{x_{i_1}}>z_{x_i}$ and such that the decision for $x_{i_1}$ is affected by the switch. In its turn such a variable exist
if either $i_1=i_0, z_{i_1}=z_{i_0}>z_{i}$ and $x_{i_0}\in B(x_i,r)$ (in particular $x_{i_0}\rightsquigarrow x_i$),
or if there exists $x_{i_2}\in B(x_{i_1},r)$ such that
$z_{i_2}>z_{i_1}$ and $x_{i_2}$ is affected by the switch. In this case again $x_{i_2}\rightsquigarrow x_i$.
Continuing, we see that in order for $x_i$ to be affected by the switch, it must be the case that $x_{i_0}\rightsquigarrow x_i$.
\end{proof}

We now obtain a probabilistic bound on the size of a largest in cardinality of the influence range classes $\mathcal{IR}_{x_i}, 1\le i\le n$.
\begin{prop}\label{prop:MaximumIR}
The following holds
\begin{align*}
\lim_{n\rightarrow\infty}\pr(\max_{1\le i\le n}|\mathcal{IR}_{x_i}|\ge n^{1\over 3})=0.
\end{align*}
\end{prop}
The choice of exponent $1/3$ is somewhat arbitrary here. In fact the bound holds for any exponent in $(0,1)$, and for our purposes,
as we are about to see in Section~\ref{section:ProofOfMainTheorem}, any constant in the range $(0,1/2)$ suffices.

\begin{proof}
Fix a variable $x$ in $\bfPhi(n,dn)$. We first establish an upper bound on the number of variables in a neighborhood $B(x,t)$ of $x$ in
the node-to-node graph $\G(\bfPhi(n,dn))$ when $t$ is moderately growing.

\begin{lemma}\label{lemma:BoundBxr}
There exists $\delta>0$ and $\epsilon=\epsilon(\delta)<1/3$ such that for all sufficiently large  $n$
\begin{align*}
\pr(|B(x,t)|\ge n^{\epsilon})\le {1\over n^2},
\end{align*}
when $t\le \delta \ln n$.
\end{lemma}
From the proof below it will be clear that the bound $1/n^2$ is very crude and in fact a bound $\exp(-n^{\epsilon/5})$ can be established. But a cruder bound suffices for
our purposes.
\begin{proof}
It is well known that for small enough $\delta>0$ and $t=\delta\ln n$,
the $B(x,t)$ is distributed approximately as a Poisson branching process with the off-spring distribution
being Poisson with parameter $\beta\triangleq dK$. Furthermore, by increasing the number of clauses by $o(n)$ the Poisson branching process stochastically
dominates the distribution of $B(x,t)$. Thus we will obtain instead an upper bound on the number of off-springs in the $t$ generations of a Poisson branching process
with parameter $\beta$. Letting $W_l$ denote the size of the $l$-th generation, our goal is then to obtain a bound on $\sum_{l\le t}W_l$. We claim
that for some $\epsilon=\epsilon(\delta)$ satisfying $\epsilon(\delta)\rightarrow 0$ as $\delta\rightarrow 0$,
the following
upper bound holds for  each $W_l, l\le t=\delta \ln n$:
\begin{align}\label{eq:TailBranching}
\pr(W_l>n^{\epsilon/2})\le \exp(-n^{\epsilon/4}),
\end{align}
from which the claim of the lemma follows by a union bound. To establish this bound we rely on the following known representation of the probability generating
function of $W_l$. That is, let $G(\theta)=\E[\theta^{W_1}]$ for $\theta>0$, where $W_1$ has Poisson mean $\beta$ distribution. Then $G(\theta)=\exp(\beta \theta-\beta)$
and $\E[\theta^{W_l}]=G^{(l)}(\theta)$ - the $l$-th iterate of function $G(\theta)$. Now we let $\theta=1+{1\over (e\beta)^{t}}$. Define $\gamma_l=1/(e\beta)^l$.
We now obtain an upper bound on  $G^{(l)}(\theta)$. We have
\begin{align*}
G^{(1)}(\theta)=\exp(\beta\theta-\beta)=\exp(\beta \gamma_t)\le 1+\gamma_{t-1},
\end{align*}
where we have used $\beta\gamma_t<1$ and inequality $e^z\le 1+ez$ for $z\le 1$. Then
\begin{align*}
G^{(2)}(\theta)=\exp(\beta G^{(1)}(\theta)-\beta)\le \exp(\beta \gamma_{t-1})\le 1+\gamma_{t-2},
\end{align*}
since $\beta\gamma_{t-1}<1$. Continuing, we obtain $G^{(l)}(\theta)\le 1+\gamma_{t-l}, 1\le l\le t$. Applying this bound
\begin{align*}
\pr(W_l\ge n^{\epsilon/2})&=\pr(\theta^{W_l}\ge \theta^{n^{\epsilon/2}})\\
&\le \theta^{-n^{\epsilon/2}}\E[\theta^{W_l}]\\
&\le \theta^{-n^{\epsilon/2}}(1+\gamma_{t-l})\\
&\le 2\theta^{-n^{\epsilon/2}}.
\end{align*}
Now
\begin{align*}
\theta^{-n^{\epsilon/2}}&=\exp(-n^{\epsilon/2}\ln(\theta)) \\
&=\exp(-n^{\epsilon/2}(\gamma_t+o(\gamma_t)).
\end{align*}
Now since $t=\delta \ln n$, then $\gamma_t=(e\beta)^{-t}=n^{-\ln(e\beta)\delta}$, implying the upper bound $\exp(-n^{\epsilon/4})$ for large enough $n$
when $\epsilon(\delta)>2\ln(e\beta)\delta$. This completes  the proof of the
bound (\ref{eq:TailBranching}) and of  lemma.
\end{proof}

Now we complete the proof of the Proposition. Applying union bound we have that that for every $\epsilon>0$, $|B(x_i,t)|\le n^{\epsilon})$
for all $i=1,\ldots,n$ with probability approaching unity as $n\rightarrow\infty$.
Given two variables $x_i$ and $x_j$ if $x_i\rightsquigarrow x_j$ and the distance in $\G(\bfPhi(n,dn))$ between
$x_i$ and $x_j$ is at least $t$, then there must exist $x_k\in B(x_i,t)\setminus B(x_i,t-1)$ such that $x_i\rightsquigarrow x_k$.
Given a sequence $y_0=x_i,y_1,\ldots,y_t=x_k$, with $x_k\in B(x_i,t)\setminus B(x_i,t-1)$, the probability of
an event $Z_{y_l}>Z_{y_{l+1}}, 0\le l\le t-1$ is $1/(t+1)!$. The total number of paths between $x_i$ and variables in $B(x_i,t)\setminus B(x_i,t-1)$
is trivially at most $B(x_i,t)$, since $B(x_i,t)$ is tree. Thus, conditioned on $B(x_i,t)$, the expected number of variables in $B(x_i,t)$
is at most $B(x_i,t)r^t/(t+1)!$, where the extra factor $r^t$ is due to choices of points $y_1,\ldots,y_t$ on a given path. When $t=\epsilon \ln$,
the expected number is $B(x_i,t)n^{-\Omega(\ln\ln n)}$. Applying the bound Lemma~\ref{lemma:BoundBxr} and a union bound over $x_i$ we obtain the result.
\end{proof}

\section{The clustering property of random NAE-$K$-SAT problem}\label{section:Clustering}
In this section we establish the clustering property of random NAE-$K$-SAT problem when $d$ is large enough (in terms of
$K$).
Recall that the random NAE-$K$-SAT formula $\bfPhi(n,dn)$ is satisfiable with probability approaching unity as $n\rightarrow\infty$,
when $d\le d_s$, where $d_s=2^{K-1}\ln 2-\ln 2/2-1/4-f(K)$ for some function $f(K)$ satisfying $\lim_{K\rightarrow\infty}f(K)=0$.
Recalling our notation for the set of satisfying assignment $\mathbb{SAT}(\Phi)$ of a formula $\Phi$, we have
$\pr(\mathbb{SAT}(\bfPhi(n,dn))\ne \emptyset)\rightarrow 1$ as $n\rightarrow\infty$ for every $d<d_s$.

The notion of ``clustering'' we consider is with respect to the Hamming
distance where
the Hamming distance between two assignments $\sigma^1$ and $\sigma^2$,
denoted $\rho(\sigma^1,\sigma^2)$,
is the number of variables $x_i$ with different assignments according to $\sigma^1$ and $\sigma^2$.
A simplistic notion of clustering might say that the ``satisfaction
graph'', the graph on
satisfying assignments where two assignments are deemed adjacent if the
Hamming distance between them is $o(n)$, has many connected components.
A condition in turn that implies this simple notion is that for every pair
of satisfying assignment $\sigma^1$ and $\sigma^2$, it is the case that
$\rho(\sigma^1,\sigma^2)/n \not\in (\beta-\eta,\beta)$ for some $\eta > 0$.
Note that this implies that for any pair of satisfying assignments
$\sigma^1$ and $\sigma^3$ with $\rho(\sigma^1,\sigma^3) > \beta n$ they must be
disconnected in the satisfaction graph, or else there will be a point $\sigma^2$ on
the path between them with $\rho(\sigma^1,\sigma^2)/n \in (\beta-\eta,\beta)$.

Working purely with this notion we only get a clustering result for very high
densities $d$, specifically for $d$ at least $d_s/2$. (We skip details since it is not needed for our main result).
To get a result for smaller
densities we work with a more sophisticated notion of clustering inspired by
\cite{rahman2014local}. Informally, this notion may be seen to allow paths between any pair
of vertices in the graph on satisfying assignments mentioned above. However
(again informally) it restricts the number of ``fundamentally'' different
paths to be small. Formally, we insist that there can not be many satisfying assigments
$\sigma^1,\ldots,\sigma^m$ with all pairwise distance being between
$(\beta-\eta)n$ and $\beta n$. We give the formal definition below.

Fix  $\beta,\eta\in [0,1]$ and a positive integer $m$. Given an NAE-$K$-SAT formula $\Phi$, denote by $\mathbb{SAT}(\Phi;\beta,\eta,m)$
the set of all $m$-tuples $(\sigma^1,\ldots,\sigma^m)$ of assignments $\sigma^j:\{x_1,\ldots,x_n\} \to \{0,1\}, ~1\le j\le m$ satisfying the following properties:
\begin{enumerate}
\item[(a)] Every $\sigma^j,~1\le j\le m$ is a satisfying assignment. Namely $\mathbb{SAT}(\Phi;\beta,\eta,m)\subset \mathbb{SAT}^m(\Phi)$.
\item[(b)] For every $j,k$, $(\beta-\eta)n\le \rho(\sigma^j,\sigma^k)\le \beta n$.
\end{enumerate}
In our application we will choose $\eta$ to be much smaller than $\beta$. In this case the pairwise distances $\rho(\sigma^j,\sigma^k)$ are nearly $\beta n$.
Thus we may think of the such an $m$-tuple as a set of $m$ equidistant points in the Hamming cube $\{0,1\}^n$ with pairwise distances nearly $\beta n$.

We now state the main result of this section. Intuitively it states for certain choices of $\beta,\eta$ and $m$ which depend on $K$ only, there
are no such $m$ equidistance points when $d$ crosses the threshold $\approx (d_s/K)\ln^2 K$.

\begin{theorem}\label{theorem:StrongClustering}
Fix arbitrary $0<\epsilon<1$, and let $\beta={\ln K\over K},\eta=\left({\ln K\over K}\right)^2$,  $m=\lceil{\epsilon^2 K\over \ln K}\rceil$.
Then there exists $K_0=K_0(\epsilon)$, such that for all $K\ge K_0$ and  $d\ge (1+\epsilon)2^{K-1}\ln^2 K/K$, the following holds
\begin{align*}
\lim_{n\rightarrow\infty}\pr\left(\mathbb{SAT}\left(\bfPhi(n,dn),\beta,\eta,m\right)=\emptyset\right)=1.
\end{align*}
\end{theorem}

\begin{proof}
The proof is based on the application of the first moment argument. We consider the expected number of $m$-tuples satisfying the conditions (a)-(b), and show that
this expectation converges to zero exponentially fast as $n\rightarrow\infty$. Applying Markov's inequality the result then will follow.

We begin by computing asymptotically the number of $m$-tuples $\sigma^1,\ldots,\sigma^m$ satisfying condition (b)  only. We have $2^n$ choices for $\sigma^1$. For any fixed choice
of $\sigma^1$, and any fixed $j=2,\ldots,m$ the number of choices for $\sigma^j$ is
\begin{align*}
\sum_{n(\beta-\eta)\le r\le n\beta}{n\choose r},
\end{align*}
by considering all the subsets of variables $x_1,\ldots,x_n$ where $\sigma^1$ and $\sigma^j$ disagree. Since this applies for every $j$, we obtain
the following upper bound on the number of $m$-tuples satisfying (b):
\begin{align*}
2^n\left(\sum_{n(\beta-\eta)\le r\le n\beta}{n\choose r}\right)^{m-1}.
\end{align*}
This bound is clearly loose, since it ignores the constraints on $\rho(\sigma^j,\sigma^k)$ for $j,k\ge 2$. Nevertheless, it suffices for our purposes.
We now obtain an asymptotic upper bound on this expression in terms of $\epsilon,K$ and $n$.

Using Stirling's approximation and since the function $-x\ln x$ is increasing in the range $x<e^{-1}$, and decreasing in the range $x>e^{-1}$, the expression is at most
\begin{align}\label{eq:mtupleCombinatorial}
\exp\left(n\ln 2-nm\beta\ln\beta-nm(1-\beta)\ln(1-\beta))+o(n)\right).
\end{align}
Here we use $\beta=\ln K/K<e^{-1}$, for sufficiently large $K$.
Further, the same asymptotics gives $-\ln\beta=\ln K+O_K(\ln\ln K)$, implying
\begin{align*}
-m\beta\ln\beta&=m\left(\beta\ln K+O_K(\ln\ln K)\right)\\
&=\epsilon^2\ln K+O_K(\ln\ln K).
\end{align*}
Next, we have for sufficiently large $K$
\begin{align*}
-m(1-\beta)\ln(1-\beta))&\le m((\ln K/K)+o_K(\ln K/K))\\
&\le \epsilon^2+o_K(1).
\end{align*}
We conclude that for sufficiently large $K$, the term (\ref{eq:mtupleCombinatorial}) is at most
\begin{align}\label{eq:mtupleCombinatorial-2}
\exp(n\epsilon^2\ln K+nO_K(\ln\ln K)+o(n)).
\end{align}

We now compute an upper bound on the probability that a given $m$-tuple $\sigma^1,\ldots,\sigma^m$ satisfying (b), consists of satisfying assignments.
Let $C$ be a clause generated uniformly at random from the space of all clauses (a generic element of the formula $\bfPhi(n,dn)$).
Applying
the truncated exclusion-inclusion principle, the probability that $C$ is satisfied by every assignment $\sigma^1,\ldots,\sigma^m$ is
\begin{align*}
\pr(C~\text{satisfied by}~\sigma^j,~\forall j=1,\ldots,m)&=1-\pr(\exists j: ~C~\text{is not satisfied by}~\sigma^j,~1\le j\le m)\\
&\le 1-\sum_{1\le j\le m}\pr(C~\text{is not satisfied by}~\sigma^j)\\
&+\sum_{1\le j_1< j_2\le m}\pr(C~\text{is not satisfied by}~\sigma^{j_1},\sigma^{j_2}).
\end{align*}
Now $\pr(C~\text{is not satisfied by}~\sigma^j)=2^{-K+1}$. Also for every two assignments $\sigma^1$ and $\sigma^2$ which disagree in $n_0\le n$ variables
\begin{align*}
\pr(C~\text{is not satisfied by}~\sigma^{1},\sigma^{2})=2^{-K+1}\left(\left({n_0\over n}\right)^K+\left(1-{n_0\over n}\right)^K\right).
\end{align*}
We conclude that for every $m$-tuple $\sigma_1,\ldots,\sigma_m$ satisfying (b)
\begin{align*}
\pr(\sigma^1,\ldots,\sigma^m\in\mathbb{SAT}(\bfPhi(n,dn))&\le \left(1-m2^{-K+1}+(m(m-1)/2)2^{-K+1}(\beta^K+(1-\beta+\eta)^K)\right)^{dn} \\
&\le \left(1-\epsilon^2K(\ln K)^{-1}2^{-K+1}+\epsilon^4K^2(\ln K)^{-2}2^{-K+2}(K^{-1}+o_K(K^{-1}))\right)^{dn}.
\end{align*}
Here we used the fact that for $\beta=\ln K/K$ and $\eta=(\ln K/K)^2$, we have
\begin{align*}
\beta^K+(1-\beta+\eta)^K=K^{-1}+o_K(K^{-1}).
\end{align*}
The upper bound then simplifies to
\begin{align*}
\left(1-\epsilon^2 K(\ln K)^{-1}2^{-K+1}+o_K(K(\ln K)^{-1}2^{-K})\right)^{dn},
\end{align*}
which applying the lower bound  $d\ge (1+\epsilon) (2^{K-1}/K)\ln^2 K$ leads to a bound
\begin{align*}
\exp\left(-n(1+\epsilon)\epsilon^2\ln K+no_K(\ln K)\right).
\end{align*}
Now combining with (\ref{eq:mtupleCombinatorial-2}),
we conclude that the expected number of $m$-tuples satisfying conditions (a) and (b) is at most
\begin{align*}
\exp(-n\epsilon^3\ln K+no_K(\ln K)),
\end{align*}
and the proof of the theorem is complete.
\end{proof}

\section{Proof of Theorem~\ref{theorem:MainResultNAE-$K$-SAT}}\label{section:ProofOfMainTheorem}
The main result of this section states that  if a $\tau$-\local\ algorithm works
well on random instances of NAE-$K$-SAT, then it can be run several times
to produce several satisfying assignments, and in particular such that their overlaps (Hamming distances) satisfy properties (a),(b) described in the previous sections with parameters
$\alpha,\eta$ and $m$ given in Theorem~\ref{theorem:StrongClustering}. Since such overlaps are ''forbidden'' by this theorem, we will obtain a contradiction.
We state our main proposition below and show how
Theorem~\ref{theorem:MainResultNAE-$K$-SAT} follows almost immediately.
The rest of this section is devoted to the proof of the proposition.

We first recall some notation from
Section~\ref{section:LongRangeIndependence}.
Given a local rule $\tau:\mathcal{SAT}_r \to [0,1]$,
let $\sigma_{\Phi,\bfZ,\bfU}$ denote the assignment
produced by the $\tau$-\local\ algorithm on input $\Phi$,
ordering given by $\bfZ$, and using $\bfU$ to determine
the rounding of the probabilities given by $\tau$.
Recall that $\rho(\sigma^1,\sigma^2)$ denotes the Hamming
distance between assigments $\sigma^1$ and $\sigma^2$.
Let $\alpha_n$ denote the probability that $\tau$-\local\ algorithm finds a satisfying assignment in a random formula $\bfPhi(n,dn)$.
Namely,
$\alpha_n=\pr(\sigma_{\bfPhi(n,dn),\bfZ,\bfU}\in\mathbb{SAT}(\bfPhi(n,dn)))$ and the claim of Theorem~\ref{theorem:MainResultNAE-$K$-SAT}
is that $\lim_n\alpha_n=0$.

\begin{proposition}
\label{prop:local-overlap}
Fix $r <\infty$ and let
$\tau:\mathcal{SAT}_r \to [0,1]$ be any balanced local rule.
Suppose $\limsup_n\alpha_n>0$. Then for every $0<\eta<\beta$ such that $[\beta-\eta,\beta]\subset [0,1/2]$
and every positive integers $m$,  $K$ and $d$,
$$\liminf_n\pr_{\bfPhi(n,dn)}\left(
\mathbb{SAT}(\bfPhi(n,dn);\beta,\eta,m)\ne \emptyset
\right)>0.$$
\end{proposition}

\begin{proof}[Proof of Theorem~\ref{theorem:MainResultNAE-$K$-SAT}]
The result follows immediately from Theorem~\ref{theorem:StrongClustering} and Proposition~\ref{prop:local-overlap} by setting $\beta,\eta$ and $m$ exactly
as in Theorem~\ref{theorem:StrongClustering} and noting that $[\beta-\eta,\beta]\subset [0,1/2]$ is satisfied for sufficiently large $K$.
\end{proof}

\subsection{Proof of \protect Proposition~\ref{prop:local-overlap}}

\begin{proof}[Proof of \protect Proposition~\ref{prop:local-overlap}]
Given a random formula $\bfPhi(n,dn)$ and a random sequence $\bfZ$ generating the order of setting the variables, let us consider
$m$ independent vectors $\bfU^0,\ldots,\bfU^{m-1}$ which can be used to generate assignments.
By definition we have
\begin{align*}
\pr(\sigma_{\bfPhi(n,dn),\bfZ,\bfU^j}\in\mathbb{SAT}(\bfPhi(n,dn)))=\alpha_n,
\end{align*}
for $j=0,\ldots,m-1$.
We now construct a sequence of vectors $\bfV^{t,j}, 0\le t\le n,0\le j\le m-1$, where for each $j=1,\ldots,m-1$, the sequence $\bfV^{t,j}$ will
interpolate between vectors $\bfU^0$ and $\bfU^j$. Specifically, let
$\bfV^{t,j}=(V_1^{t,j},\ldots,V_n^{t,j})$ where $V^{t,j}_i=U^j_i, i\le t$ and $V^{t,j}_i=U^0_i, t<i\le n$.
Note that for every $t=0,1,\ldots,n$, $\bfV^{t,j}$ is a vector of i.i.d. random variables with the uniform in $[0,1]$ distribution. Furthermore,
$\bfV^{0,j}=\bfU^0$,$\bfV^{t,0}=\bfU^0$, and $\bfV^{n,j}=\bfU^j$.
Recall the notation $\mathcal{IR}_{x_t}$ for the influence region
of variable $x_t$, i.e., all variables whose decision is potentially
influenced by be assigment of $x_t$ by the $\tau$-\local\ algorithm.
Observe that given any realizations $\bfu^j, 0\le j\le m-1$ of vectors $\bfU^j$,  and the corresponding realizations $\bfv^{t,j}$ of $\bfV^{t,j}$, we have
\begin{align}\label{eq:rhoincrements}
\rho(\sigma_{\Phi,\bfz,\bfv^{t+1,j}},\sigma_{\Phi,\bfz,\bfv^{t,j}})\le |{\mathcal{IR}}_{x_{t+1}}|,\qquad 0\le t\le n-1,
\end{align}
since, $\bfv^{t,j}$ and $\bfv^{t+1,j}$ differ only in one coordinate $t+1$, and
by Proposition~\ref{prop:InfluenceResistance} changing the value of $u_{t+1}$ impacts only the decisions for variables in
${\mathcal{IR}}_{x_{t+1}}$.
We now consider a realization $\Phi$ of a formula $\bfPhi(n,dn)$ and realization $\bfz$ of the order $\bfZ$. $\Phi$ and $\bfz$ uniquely determine
sets ${\mathcal{IR}}_{x_i}, 1\le i\le n$.
Let $\mathcal{E}_n$ denote the event (the set of $\Phi$ and $\bfz$) that
$\max_{1\le i\le n}|{\mathcal{IR}}_{x_i}|\le n^{1/3}$. By Proposition~\ref{prop:MaximumIR}
we have
\begin{align}\label{eq:Enlimitzero}
\lim_{n\rightarrow\infty}\pr(\mathcal{E}_n)=1.
\end{align}
We assume without the loss of generality that $n$ is large enough so that $n^{1/3}<(\beta-\eta)n$.
We have by property (\ref{eq:Probhalf}) of Lemma~\ref{lemma:equiprob} that for every $\Phi$ and $\bfz$,
\begin{align*}
\E[\rho(\sigma_{\Phi,\bfz,\bfU^0},\sigma_{\Phi,\bfz,\bfU^j})]=n/2,
\end{align*}
for each $j=1,\ldots,m-1$.

We first suppose that $\Phi$ and $\bfz$ are realizations such that the event $\mathcal{E}_n$ takes place. Then, we can find $t_0=t_0(\Phi,\bfz)$ such
that
\begin{align*}
\E[\rho(\sigma_{\Phi,\bfz,\bfU^0},\sigma_{\Phi,\bfz,\bfV^{t_0,j}})]\in \left[(\beta-\eta/2)n,(\beta-\eta/2)n+n^{1/3}\right],
\end{align*}
for all $j=1,\ldots,m-1$,
as by (\ref{eq:rhoincrements}) the increments $\rho(\sigma_{\Phi,\bfz,\bfV^{t+1}},\sigma_{\Phi,\bfz,\bfV^{t,j}})$ are bounded by $n^{1/3}$ with probability
one with respect to the randomness of $\bfV^{t,j}$.
Note that $t_0$ does not depend on $j$ since $\bfV^{t,j}$ are identically distributed for $1\le j\le m-1$.
Furthermore, since $\bfU^0$ and $\bfU^j$ are identical in distribution, we also have for every $0\le j_1<j_2\le m-1$
\begin{align*}
\E[\rho(\sigma_{\Phi,\bfz,\bfV^{t_0,j_1}},\sigma_{\Phi,\bfz,\bfV^{t_0,j_2}})]\in \left[(\beta-\eta/2)n,(\beta-\eta/2)n+n^{1/3}\right].
\end{align*}
We now fix $j_1\ne j_2$ and argue that in fact $\rho(\sigma_{\Phi,\bfz,\bfV^{t_0,j_1}},\sigma_{\Phi,\bfz,\bfV^{t_0,j_2}})$ is concentrated around its mean as $n\rightarrow\infty$.
The distance is a function of $n+t_0$ i.i.d. random variables $U_1^{j_1},\ldots,U_{t_0}^{j_1}; U_1^{j_2},\ldots,U_{t_0}^{j_2};U_{t_0+1}^0,\ldots,U_n^0$.
Further, changing any one of these $n+t_0$
random variables changes the distance $\rho$ by at most $2n^{1/3}$ again by Proposition~\ref{prop:InfluenceResistance} and by our assumption that $\Phi$ and $\bfz$
are realizations such that the event $\mathcal{E}_n$ holds.
Applying Azuma's inequality
\begin{align*}
\pr&\left(\Big|\rho(\sigma_{\Phi,\bfz,\bfV^{t_0,j_1}},\sigma_{\Phi,\bfz,\bfV^{t_0,j_2}})-(\beta-\eta/2)n\Big|\ge {\eta\over 4}n\right)\\
&\le 2\exp\left(-{({\eta\over 4}n-2n^{1\over 3})^2\over 2(n+t_0)n^{2\over 3}}\right)\\
&=\exp(-\delta n^{1/3}+o(n^{1\over 3})),
\end{align*}
for some constant $\delta>0$, and the concentration is established. The event
\begin{align*}
\Big|\rho(\sigma_{\Phi,\bfz,\bfV^{t_0,j_1}},\sigma_{\Phi,\bfz,\bfV^{t_0,j_2}})-(\beta-\eta/2)n\Big|< {\eta\over 4}n
\end{align*}
implies the event
\begin{align*}
\rho(\sigma_{\Phi,\bfz,\bfV^{t_0,j_1}},\sigma_{\Phi,\bfz,\bfV^{t_0,j_2}})\in [(\beta-\eta) n,\beta n].
\end{align*}
We conclude that for every $\Phi$ and $\bfz$ such that the event $\mathcal{E}_n$ takes place, we have
\begin{align}\label{eq:rhoConcentration0}
\lim_n\pr\left(\rho(\sigma_{\Phi,\bfz,\bfV^{t_0,j_1}},\sigma_{\Phi,\bfz,\bfV^{t_0,j_2}})\in [(\beta-\eta n),\beta n]\right)=1.
\end{align}
Since $m$ does not depend on $n$, we obtain by union bound
\begin{align}\label{eq:rhoConcentration}
\lim_n\pr\left(\forall~0\le j_1\ne j_2\le m-1,~\rho(\sigma_{\Phi,\bfz,\bfV^{t_0,j_1}},\sigma_{\Phi,\bfz,\bfV^{t_0,j_2}})\in [(\beta-\eta) n,\beta n]\right)=1.
\end{align}
For completion, let us set $t_0=0$ when $\Phi$ and $\bfz$ are such that the event $\mathcal{E}_n$ does not take place.
Let now $T=t_0(\bfPhi(n,dn),\bfZ)$ to be thus defined random variable. This way we have
assignments $\sigma_{\Phi,\bfz,\bfV^{T,j}}, 0\le j\le m-1$ defined for all realizations of $\Phi$ and $\bfz$, in particular whether the event $\mathcal{E}_n$
takes place or not. Since the former is the high probability event, we conclude from above that
\begin{align}\label{eq:rhoConcentration1}
\lim_n\pr\left(\forall~0\le j_1\ne j_2\le m-1,~\rho(\sigma_{\bfPhi(n,dn),\bfZ,\bfV^{T,j_1}},\sigma_{\bfPhi(n,dn),\bfZ,\bfV^{T,j_2}})\in [(\beta-\eta) n,\beta n]\right)=1.
\end{align}
Thus, we established that with high probability as $n\rightarrow\infty$ there exist
a sequence of assignments $\sigma^j\triangleq\sigma_{\bfPhi(n,dn),\bfZ,\bfV^{T,j}},~0\le j\le m-1$ satisfying property (b) of the definition of
$\mathbb{SAT}(\Phi;\beta,\eta,m)$.

Our next goal is to show that the assignments $\sigma^j, 0\le j\le m-1$ above
are also satisfying formula $\bfPhi(n,dn)$ with probability bounded away from zero as $n\rightarrow\infty$. To be exact we claim
\begin{align}\label{eq:SATPositiveProb}
\liminf_n&\pr(\sigma_{\bfPhi(n,dn),\bfZ,\bfV^{T,j}}\in\mathbb{SAT}(\Phi),~0\le j\le m-1)>0.
\end{align}
namely property (a) holds with probability bounded away from zero, and thus the set
$\mathbb{SAT}(\Phi;\beta,\eta,m)$ is non-empty with probability bounded away from zero, as claimed.
Observe that $\sigma_{\bfPhi(n,dn),\bfZ,\bfV^{T,j}}$ have identical distribution
for all $j$. Furthermore, each of them individually is  distributed as $\sigma_{\bfPhi(n,dn),\bfZ,\bfU^j}, 0\le j\le m-1$
since the random variable $T$ only affects the indices $i$ for which we switch from  $U_i^0$ vs $U_i^j$,
and since each vector $\bfU^j$ is an i.i.d. vector of random variables. Therefore,
\begin{align*}
\pr(\sigma_{\bfPhi(n,dn),\bfZ,\bfV^{T,j}}\in\mathbb{SAT}(\bfPhi(n,dn)))=\alpha_n,
\end{align*}
for each $j$.
Suppose $\Phi,\bfz$ are such that the event $\mathcal{E}_n$ takes place
and fix the corresponding deterministic value $t_0=t_0(\Phi,\bfz)$. In the derivation below we use notation $\pr_{Z}$ to indicate probability with
respect random variable $Z$. We have
\begin{eqnarray}
\lefteqn{\pr_{\bfU^0,\ldots,\bfU^{m-1}}(\sigma_{\Phi,\bfz,\bfV^{T,j}}\in\mathbb{SAT}(\Phi),~0\le j\le m-1)} \notag\\
&=& \E_{\bfU^0,\ldots,\bfU^{m-1}}[\textbf{1}(\sigma_{\Phi,\bfz,\bfV^{t_0,j}}\in\mathbb{SAT}(\Phi),~0\le j\le m-1)] \notag\\
&=&\E_{U^0_{t_0+1},\ldots,U^0_n}[\E_{U^j_i, 1\le i\le t_0,1\le j\le m-1}
[\textbf{1}(\sigma_{\Phi,\bfz,\bfV^{t_0,j}}\in\mathbb{SAT}(\Phi)),~0\le j\le m-1~\big|~U^0_{t_0+1},\ldots,U^0_n]] \notag\\
&\stackrel{(a)}{=}&
\E_{U^0_{t_0+1},\ldots,U^0_n}[\E_{U^0_1,\ldots,U^0_{t_0}}^{m}[
\textbf{1}(\sigma_{\Phi,\bfz,\bfU^0}\in\mathbb{SAT}(\Phi))~\big|~U^0_{t_0+1},\ldots,U^0_n]] \notag\\
&\stackrel{(b)}{\ge}&
\E^{m}_{U^0_{t_0+1},\ldots,U^0_n}[\E_{U^0_1,\ldots,U^0_{t_0}}[
\textbf{1}(\sigma_{\Phi,\bfz,\bfU^0}\in\mathbb{SAT}(\Phi))~\big|~U^0_{t_0+1},\ldots,U^0_n]] \notag\\
&=& \E^m_{\bfU^0}[\textbf{1}(\sigma_{\Phi,\bfz,\bfU^0}\in\mathbb{SAT}(\Phi))~\big] \notag\\
&=&\pr^m_{\bfU^0}(\sigma_{\Phi,\bfz,\bfU^0}\in\mathbb{SAT}(\Phi)) \label{eq:Jensen}.
\end{eqnarray}
Here (a) follows since $\bfU^j$ are independent vectors of i.i.d. random variables and (b) follows by applying Jensen's inequality and the convexity of the polynomial
function $t^m$ on $t\in [0,\infty)$ for all positive integers $m$.

Suppose now $\Phi$ and $\bfz$ are such that the event $\mathcal{E}_n$ does not take place. Then
$\sigma_{\Phi,\bfz,\bfV^{T,j}}=\sigma_{\Phi,\bfz,\bfU^0}$, implying
\begin{align*}
\pr_{\bfU^0,\ldots,\bfU^{m-1}}(\sigma_{\Phi,\bfz,\bfV^{T,j}}\in\mathbb{SAT}(\Phi),~0\le j\le m-1)
&=\pr_{\bfU^0}(\sigma_{\Phi,\bfz,\bfU^0}\in\mathbb{SAT}(\Phi))\\
&\ge\pr^m_{\bfU^0}(\sigma_{\Phi,\bfz,\bfU^0}\in\mathbb{SAT}(\Phi)).
\end{align*}
Combining with (\ref{eq:Jensen}) we conclude that for every $\Phi,\bfz$ we have
\begin{align*}
\pr_{\bfU^0,\ldots,\bfU^{m-1}}(\sigma_{\Phi,\bfz,\bfV^{T,j}}\in\mathbb{SAT}(\Phi),~0\le j\le m-1)
\ge \pr^m_{\bfU^0}(\sigma_{\Phi,\bfz,\bfU^0}\in\mathbb{SAT}(\Phi)).
\end{align*}
Since $\pr_{\bfU^0}(\sigma_{\Phi,\bfz,\bfU^0}\in\mathbb{SAT}(\Phi))=\alpha_n$, then integrating over $\bfPhi(n,dn)$ and $\bfZ$,
we obtain
\begin{align*}
\pr(\sigma_{\bfPhi(n,dn),\bfZ,\bfV^{T,j}}\in\mathbb{SAT}(\bfPhi(n,dn)),~0\le j\le m-1)&\ge \pr(\sigma_{\bfPhi(n,dn),\bfZ,\bfU^0}\in\mathbb{SAT}(\bfPhi(n,dn))\\
&=\alpha_n^m,
\end{align*}
implying
\begin{align*}
\liminf_n&\pr(\sigma_{\bfPhi(n,dn),\bfZ,\bfV^{T,j}}\in\mathbb{SAT}(\bfPhi(n,dn)),~0\le j\le m-1)\\
&\ge \liminf_n \alpha_n^m\\
&>0,
\end{align*}
and (\ref{eq:SATPositiveProb}) is established.

\end{proof}


\section*{Acknowledgements}
The authors gratefully acknowledge many enlightening conversations with Federico Ricci-Tersenghi, Riccardo Zecchina, Marc Mezard, Giorgio Parisi, Florent Krzakala, Lenka Zdeborova,
and many others who generously provided constructive feedback to the earlier version of this paper.

\bibliographystyle{amsalpha}

\newcommand{\etalchar}[1]{$^{#1}$}
\providecommand{\bysame}{\leavevmode\hbox to3em{\hrulefill}\thinspace}
\providecommand{\MR}{\relax\ifhmode\unskip\space\fi MR }
\providecommand{\MRhref}[2]{%
  \href{http://www.ams.org/mathscinet-getitem?mr=#1}{#2}
}
\providecommand{\href}[2]{#2}

\end{document}